\documentclass[onefignum,onetabnum]{siamonline250211}

\usepackage{amssymb}
\usepackage{mathrsfs}
\usepackage{stmaryrd}

\SetSymbolFont{stmry}{bold}{U}{stmry}{m}{n}
  
\usepackage{mathtools}
\usepackage{bm}
\usepackage{accents}
\usepackage{relsize}

\usepackage{multirow}
\usepackage{booktabs}

\usepackage{algpseudocode}

\usepackage{enumerate}
\usepackage[normalem]{ulem}

\definecolor{ForestGreen}{RGB}{34,139,34}

\allowdisplaybreaks

\usepackage{mathtools}

\usepackage[normalem]{ulem} 

\newcommand{\va}{{\mathbf{a}}}

\newcommand{\ve}{{\mathbf{e}}}

\newcommand{\vj}{{\mathbf{j}}}

\newcommand{\vm}{{\mathbf{m}}}
\newcommand{\vn}{{\mathbf{n}}}

\newcommand{\vv}{{\mathbf{v}}}

\newcommand{\vx}{{\mathbf{x}}}

\newcommand{\vJ}{{\mathbf{J}}}

\newcommand{\vM}{{\mathbf{M}}}

\newcommand{\vP}{{\mathbf{P}}}

\newcommand{\vU}{{\mathbf{U}}}

\newcommand{\cL}{{\mathcal{L}}}

\newcommand{\cP}{{\mathcal{P}}}

\newcommand{\cX}{{\mathcal{X}}}

\newcommand{\RR}{\mathbb{R}} 
\newcommand{\vzero}{\mathbf{0}} 

\newcommand{\st}{\mbox{ s.t. }}

\DeclareMathOperator*{\argmin}{arg\,min} 

\newcommand{\bc}{\begin{center}}
\newcommand{\ec}{\end{center}}

\newcommand{\bdm}{\begin{displaymath}}
\newcommand{\edm}{\end{displaymath}}

\newcommand{\beq}{\begin{equation}}
\newcommand{\eeq}{\end{equation}}

\newcommand{\bfl}{\begin{flushleft}}
\newcommand{\efl}{\end{flushleft}}

\newcommand{\bt}{\begin{tabbing}}
\newcommand{\et}{\end{tabbing}}

\newcommand{\beqn}{\begin{eqnarray}}
\newcommand{\eeqn}{\end{eqnarray}}

\newcommand{\beqs}{\begin{align*}} 
\newcommand{\eeqs}{\end{align*}}  

\newsiamremark{remark}{Remark}

\ifpdf 
\hypersetup{
pdftitle={Scalable Dynamic Optimal Transport via Distributed Linearized ADMM},
pdfauthor={Hari Dahal, Rongjie Lai, and Yangyang Xu}
}

\title{Scalable Dynamic Optimal Transport via Distributed Linearized ADMM}
\author{Hari Dahal\thanks{Department of Mathematical Sciences, Rensselaer Polytechnic Institute, Troy, NY 12180, (\email{haridahal156@gmail.com}, \email{xuy21@rpi.edu}). }
\and Rongjie Lai\thanks{Department of Mathematics, Purdue University, West Lafayette, IN 47907, (\email{lairj@purdue.edu}).}
\and Yangyang Xu\footnotemark[1]}

\date{\today}

\begin{document}

\maketitle

\begin{abstract}
In this paper, we address two fundamental challenges in the numerical solution of dynamic optimal transport (OT) problems. The first challenge arises when the initial and/or terminal densities approach zero and no positive lower bound is available. In this regime, conventional methods may become unstable or computationally inefficient, since the Lipschitz constant of the objective can scale like the reciprocal of the cube of the density. As a result, near-zero regions may lead to slow convergence or even divergence. The second challenge concerns the substantial memory cost of the dynamic formulation, whose discretization over fine spatial and temporal grids requires storing variables across the entire space-time domain. This storage burden quickly becomes prohibitive as the grid is refined or the spatial dimension increases.
To overcome the first difficulty, we reformulate the classical discretized dynamic OT problem so that the resulting objective admits an exact proximal mapping. When combined with a linearized alternating direction method of multipliers (LADMM), this reformulation yields an efficient and robust algorithm that remains stable even in challenging settings where the density may vanish. To reduce the memory burden, we further introduce a distributed formulation in which the optimization variables are partitioned across multiple agents. This design substantially lowers the storage requirement for each agent and can also lead to computational acceleration. We validate the proposed framework through numerical experiments in one- and two-dimensional spatial settings under varying levels of difficulty. The results demonstrate the stability, robustness, and scalability of the proposed method in comparison with conventional approaches.
\end{abstract}



\section{Introduction}


Optimal transport (OT) provides a principled framework for comparing and transforming probability distributions  and has become an important tool in applications across a wide range of fields, including economics \cite{buttazzo2002optimal,carlier2005variational,chiappori2010hedonic,galichon2018optimal}, image processing \cite{chizat2017unbalanced,chizat2018interpolating,papadakis2014optimal}, machine learning and generative AI \cite{altschuler2017near,kolouri2017optimal,rubner1998metric,arjovsky2017wasserstein,salimans2018improving}, probability theory \cite{rachev2006mass}, and partial differential equations \cite{carrillo2010numerical,jordan1998variational,cances2017numerical}. Comprehensive treatments of OT and its applications can be found in \cite{villani2021topics,villani2009optimal,santambrogio2015optimal,peyre2019computational}.

The foundations of OT can be traced back to the Monge mapping problem, introduced by Gaspard Monge in 1781. Monge's original formulation arose from a practical problem: to find the optimal way to transport a pile of material from one location to another 
by minimizing the cost of transportation. In a more modern framework, Monge's problem was generalized to probability measures by Leonid Kantorovich in 1942 \cite{kantorovich1942transfer}, leading to what is now known as the Monge-Kantorovich (MK) problem. In this paper, we focus on the Benamou–Brenier formulation of OT given below \cite{benamou1999numerical,benamou2000computational, maas2015generalized, lombardi2015eulerian, piccoli2014generalized}:
\begin{align}\label{problem: main}
    \min_{(\rho, \vm) \in \mathcal{S}(\bm{\rho}_0, \bm{\rho}_1) } \int_0^1 \int_{\Omega} F(\rho(t,\vx), \vm(t, \vx)) \text{d}\vx \text{d}t. 
\end{align} 
 Here, $\Omega\subset\mathbb{R}^D$, \(\rho(t,\vx)\) denotes the density field, and \(\vm(t,\vx):=\rho(t,\vx)\vv(t,\vx)\) is the flux field associated with the velocity field \(\vv(t,\vx)\).  
The dynamic cost function $F: \mathbb{R}^+ \times \mathbb{R}^D \rightarrow \overline{\mathbb{R}} := \mathbb{R} \cup \{\infty\}$ is defined as \cite{benamou1999numerical} 
    \begin{align}\label{eq: maincostfunction}
    F(b, \va) = \begin{cases}
        \displaystyle\frac{\|\va \|^2}{2b} & \text{ if } b > 0, \\
        0 & \text{ if } b = 0, \va = \mathbf{0},\\
         +\infty & \text{ if } b = 0, \va \neq \mathbf{0}.
    \end{cases}
    \end{align}
    The admissible set $\mathcal{S}(\bm{\rho}_0,\bm{\rho}_1)$ consists of all density–momentum pairs satisfying the continuity equation, the no-flux boundary condition, and the prescribed initial and terminal densities:
\begin{align}\label{cons: main}
    \mathcal{S}(\bm{\rho}_0, \bm{\rho}_1) :=  \big\{(\rho, \vm)\,|\,  & \partial_t \rho + \text{div}_x \vm = 0, \\
    \quad \quad \quad & \vm \cdot \vn = 0 \text{ for } \vx \in \partial \Omega, \rho(0, \cdot) = \bm{\rho}_0, \rho(1, \cdot) = \bm{\rho}_1\big\}, \notag
\end{align}
where \(\vn\) denotes the outward unit normal vector on $\partial\Omega$.
This dynamic formulation has been widely used in scientific computing, image processing, fluid mechanics, machine learning, and related areas, due to its close connection with PDE-constrained optimization and its ability to describe the full transport path between two distributions.

Several prominent approaches for solving Problem~\eqref{problem: main} include the augmented Lagrangian method (ALM)~\cite{benamou2000computational}, the fast iterative soft-thresholding algorithm (FISTA)~\cite{yu2024fast}, and G-prox\\~\cite{jacobs2019solving}. In the regime where the initial and/or terminal densities approach zero, FISTA may exhibit very slow convergence due to the large local Lipschitz constant of the gradient of the dynamic OT energy. The primal subproblems arising in G-prox for discretizations of Problem~\eqref{problem: main} generally do not admit closed-form solutions. Therefore, they must either be solved approximately or handled by an additional iterative solver, which can increase the overall computational cost. Moreover, both G-prox and ALM are infeasible methods in the sense that their iterates do not necessarily satisfy the mass-conservation constraint at each iteration. As a result, they may require many iterations and thus are inefficient to reduce feasibility violations to a low tolerance. 

Besides the lack of efficient and robust algorithms for handling the above mentioned challenging case, 
another difficulty arises from the 
fine-grid discretization of the density and flux fields, which requires storing very high-dimension variables over the entire space–time domain. More precisely, as the spatial grid is refined, the number of time steps increases, or the spatial dimension grows, this storage requirement quickly becomes prohibitive.

These limitations/challenges motivate the development of an efficient numerical method for dynamic OT that should be not only stable with respect to near-zero densities but also scalable in memory and computation. 
To achieve this goal, we develop an exact-proximal linearized alternating direction method of multipliers (LADMM) for the discretized dynamic OT problem. The key idea is to reformulate the classical discretized formulation by introducing auxiliary variables, so that the subproblem involving the dynamic OT energy admits an exact proximal mapping. This exact-proximal update avoids the instability associated with gradient-based updates 
near zero densities. Meanwhile, the algorithm preserves the mass-conservation constraint throughout the iterations, by performing explicit projection steps. 
To further reduce the memory burden, we introduce a distributed formulation of the proposed method. In this formulation, the optimization variables are partitioned across multiple agents, so that each agent only stores and updates a portion of the full space--time variables. This substantially reduces the per-agent storage requirement and naturally enables parallel computation.

The main contributions of this paper are summarized as follows:

\begin{enumerate}

    \item We propose an exact-proximal based LADMM for the discretized dynamic OT problem. By reformulating the problem so that the proximal mapping of the dynamic OT energy can be computed exactly, the proposed method avoids gradient-based updates that become unstable when the density approaches zero. This addresses a key limitation of projected gradient and existing primal--dual methods whose convergence may deteriorate due to the large local gradient Lipschitz constant of the dynamic OT objective.

    \item To handle very large-scale OT problems, we further develop a distributed implementation of the proposed LADMM method by partitioning the space--time variables across multiple agents. This reduces the memory requirement for each agent and enables parallel computation.

    \item We validate the proposed framework through extensive numerical experiments in one- and two-dimensional spatial settings under varying levels of difficulty and/or a distributed setting. The results demonstrate that the proposed  LADMM exhibits significantly higher 
    numerical robustness and efficiency than existing methods in challenging regimes involving near-zero densities.  For the case with distributed computing, our method is able to achieve high speed up, with over 6x speed up by using 12 CPU cores. 

\end{enumerate}


\subsection{Related work}

Many numerical approaches have been developed for solving the dynamic OT problem \eqref{problem: main}. The augmented Lagrangian method of Benamou and Brenier~\cite{benamou2000computational} pioneered the numerical solution of the dynamic formulation and has inspired many subsequent variants~\cite{lombardi2015eulerian,papadakis2014optimal}. ALM can
often converge, but it can struggle with reducing the feasibility. 

Projected gradient-type methods form another important class of algorithms. Yu et al.~\cite{yu2024fast} proposed a fast iterative soft-thresholding algorithm (FISTA) for mean-field planning problems, of which \eqref{problem: main} is a special case. Such methods can be efficient when the objective has a moderate gradient Lipschitz constant, but they become less effective in the presence of near-zero densities because of the large local curvature of the kinetic energy. In particular, since the local gradient Lipschitz constant can scale on the order of \(1/\rho^3\), projected gradient methods may require very small step sizes when \(\rho\) approaches zero, leading to degraded convergence.

Primal--dual methods provide a third class of approaches. Jacobs et al.~\cite{jacobs2019solving} proposed the G-prox method for large-scale optimal transport problems. Related primal--dual formulations for dynamic OT have also been studied by Henry et al.~\cite{henry2019primal}, where the Helmholtz--Hodge decomposition~\cite{girault2012finite} is used to enforce divergence-free constraints during the iteration. These methods
require a substantial number of iterations to reduce constraint violations, as
it is not based on an exact projection step.

In addition to the above optimization-based approaches, several learning-based methods have recently been proposed for dynamic OT and related flow problems~\cite{ruthotto2020machine, huang2023bridging, fan2023neural, wan2023scalable, huang2025unsupervised}. There has also been growing interest in dynamic OT on discrete surfaces and geometric domains~\cite{lavenant2018dynamical, dong2025gradient, yu2023computational}. These directions are complementary to the present work, which focuses on robust and scalable optimization methods for the Eulerian dynamic formulation \eqref{problem: main}.

Distributed and parallel methods for OT have been studied in several related settings. For entropic OT, Bonafini and Schmitzer~\cite{bonafini2021domain} developed a domain decomposition algorithm based on the earlier domain decomposition framework of Benamou~\cite{benamou1995domain}. Wang et al.~\cite{wang2023decentralized} proposed a decentralized entropic OT method for privacy-preserving distributed distribution comparison, using a mini-batch randomized block-coordinate descent algorithm to minimize a decentralized dual formulation. Li et al.~\cite{li2018parallel} developed a parallel primal--dual algorithm for computing the Earth Mover's Distance by reformulating the problem as an \(L_1\)-type minimization problem with an additional quadratic regularization term. However, these methods primarily address the static Monge--Kantorovich problem, entropic OT, or related variants motivated by specific applications. They do not directly solve the dynamic OT formulation \eqref{problem: main} considered in this work.

\subsection{Notations}\label{sec:notation}
We define $\mathbb{R}_+ := \{a \in \mathbb{R} \mid a \geq 0 \}$ and  
$[N]:=\{1,2,\ldots,N\}$. 
$\partial \Omega$ denotes the boundary of the set $\Omega$. {$\vM_\vj$ denotes the entry 
of $\vM$ specified by the index vector $\vj$ and $\ve_d$ denotes a vector with $1$ in the $d$-th entry and $0$ everywhere else. For a linear operator $A$, its norm 
is defined as $\|A\|_* := \sup_{\|\vx\|=1} \|A(\vx)\|$, where $\|\cdot\|$ denotes the Euclidean norm. 
The $d$-mode 
product \cite{kolda2009tensor} of a tensor $\cX \in \RR^{n_0 \times n_1 \times \ldots \times n_D}$ with a matrix $\vU \in \RR^{\hat{n} \times n_d}$ is denoted by $\cX \times_{d} \vU$ with entries given by 
\begin{align*}
    (\cX \times_d \vU)_{i_0\cdots i_{d-1}ji_{d+1}\cdots i_D} = \sum_{i_d=0}^{n_d} x_{i_0i_1\cdots i_D}u_{ji_d}.
\end{align*}
We use $\langle \cdot, \cdot\rangle$ for standard Euclidean inner product. For any positive number $\Delta$, we denote 
\(
\langle x, y \rangle_{\Delta} = \Delta \langle x, y\rangle.
\) Also, we define $\|x\|_\Delta^2 = \langle x, x \rangle_{\Delta}$.
}

\section{Proposed Algorithm}

In this section, we give the LADMM for Problem~\eqref{problem: main} with detailed derivations for its updates. 
We begin by presenting a discretized formulation of Problem~\eqref{problem: main}. 
Then, we {derive the proximal mapping of $F$} defined in \eqref{eq: maincostfunction}, which plays a key role in enabling more robust and stable performance of  
our algorithm. 
Finally, we present the complete LADMM algorithm.

\subsection{Discretization}\label{sec: discretize}
We follow the staggered grid used by Yu et al. \cite{yu2024fast} to build a discretized version of Problem~\eqref{problem: main}. Suppose that $\Omega = [0,1]^D$. We consider a uniform grid with $n_0$ segments on the time interval $[0,1]$ and $n_d$ segments on the $d$-th space dimension for $d=1,\ldots,D$. The mesh size on each dimension is $\Delta_d = \frac{1}{n_d}$ for $ d = 0, 1, \ldots, D$. In addition, for each $ d = 0, 1, \ldots, D$, we denote 
\begin{align*}
    J_d &:= \left( \frac{3}{2}, \frac{5}{2},\ldots, n_d - \frac{1}{2} \right),\\
    \overline{J}_d &:= (1, 2,\ldots, n_d ),\\
    \vJ_d &:= \overline{J}_0 \times \overline{J}_1 \times \cdots \times \overline{J}_{d-1}
    \times J_d \times \overline{J}_{d+1} \times \cdots \times \overline{J}_D,\\
    \overline{\vJ} &:= \overline{J}_0 \times \overline{J}_1 \times \cdots \times \overline{J}_D,\\
    \widehat{J}_d &:= (1,2,\ldots,n_d-1),\\
    \widehat{\vJ}_d &:= \overline{J}_0 \times \overline{J}_1 \times \cdots \times \overline{J}_{d-1}
    \times \widehat{J}_d \times \overline{J}_{d+1} \times \cdots \times \overline{J}_D, \\
    \mathbb{V}_d &:= \mathbb{R}^{n_0 \times n_1 \times\ldots n_{d-1} \times (n_d-1) \times n_{d+1} \times\ldots \times n_D}, \\
    \overline{\mathbb{V}} &:= \mathbb{R}^{n_0 \times n_1 \times\ldots \times n_D}.
\end{align*}
Also, we denote $M_0$ and $M$ respectively as the discretization of $\rho$ and $\vm$ on grid points. Let $\widehat{\vx} = (t, \vx)$. Then we define 
\begin{equation}
\begin{aligned}\label{eq:discrePM}
    M_0 &\in \mathbb{V}_0 
    \text{ with } (M_0){_\vj} = \rho(\widehat{\vx}_{\vj+\frac{1}{2}\ve_d}), \, \forall\, \vj \in \hat{\vJ}_0, \\
    M_d &\in \mathbb{V}_d 
    \text{ with } (M_d)_\vj = \vm(\widehat{\vx}_{\vj + \frac{1}{2}\ve_d}), \, \forall\, \vj \in \hat{\vJ}_d, \  \forall\, d = 1, 2, \ldots, D, \\
    \vM &:= 
    \left(M_1,\ldots, M_D \right) \in \mathbb{V}_1 \times \mathbb{V}_2 \times\ldots \mathbb{V}_D.
\end{aligned}
\end{equation}

\begin{figure}[hpbt]
\begin{center}
\includegraphics[width=0.95\textwidth]{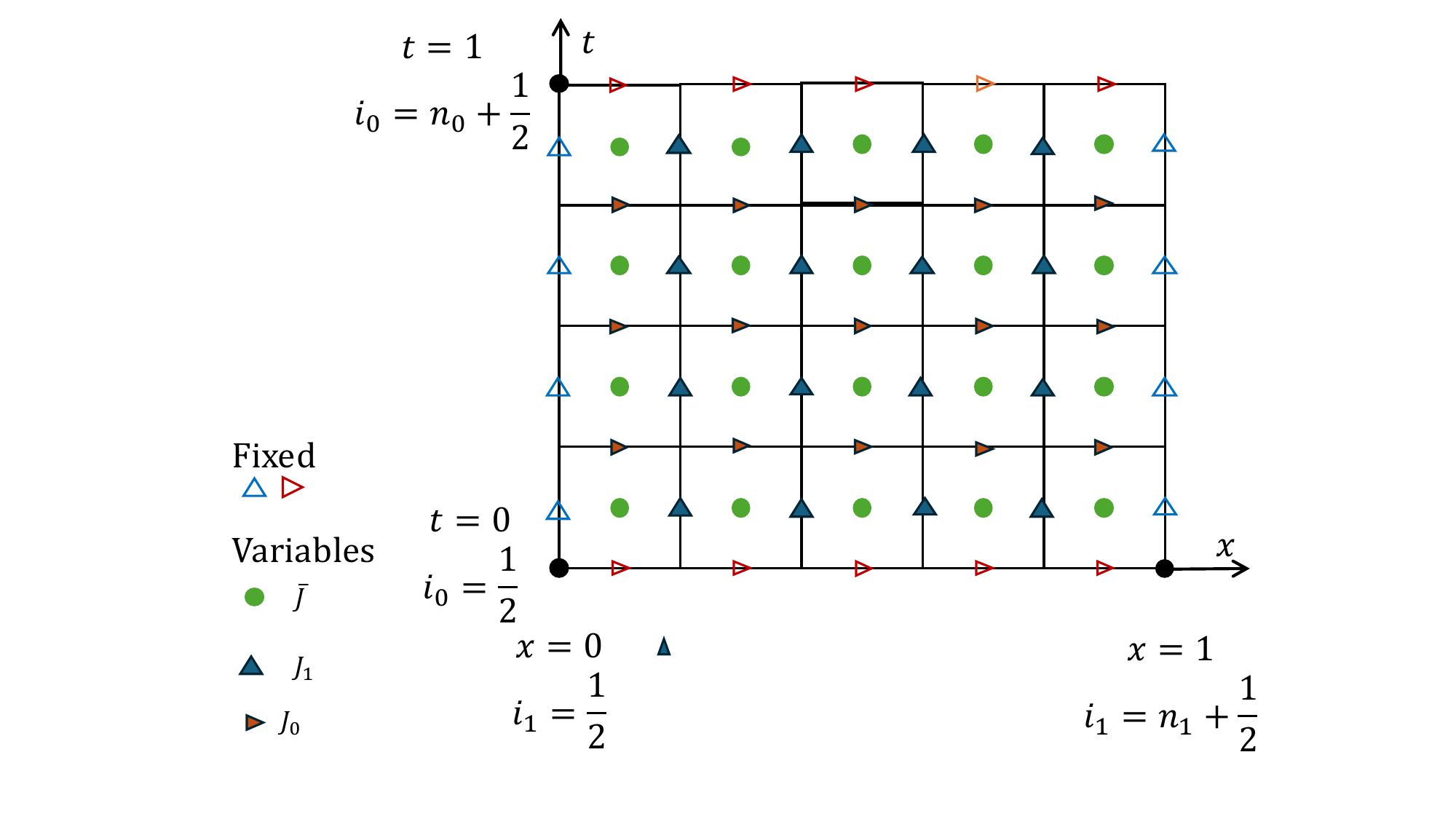}
\end{center}
\vspace{-4mm}
\caption{
Illustration of the decomposition of the time and spatial domain for $D=1$.}
\label{fig: discretization}
\end{figure}

For better understanding {of} our notations, we use the case of $D=1$ to explain them.
The index sets $\vJ_0$ and $\vJ_1$ correspond to the filled red and blue triangles in Figure~\ref{fig: discretization}, respectively. The function $\rho$ is defined on $\vJ_0$, while $\vm$ is defined on $\vJ_1$. Since $M_0$ and $M$ represent 
$\rho$ and $\vm$ at discretized points, 
we observe that 
they are defined on two staggered grids; that is, the red and 
blue triangles represent distinct spatial locations. When computing the objective in Problem~\eqref{problem: main}, 
we will evaluate the numerator and denominator at the same spatial locations. For this purpose, the averaged variables $\widetilde{M}_0$ and $\widetilde{\vM}$ that we introduce in the next subsection 
are defined on the central grid represented by the index set $\overline{\vJ}$, corresponding to the green points in Figure~\ref{fig: discretization}.

\paragraph{Objective value:}\label{subsec:objdisc}

To 
discretize the objective, we define average operators as: 
\begin{equation}
\begin{aligned}\label{eq:Avgop}
    &\text{for each } d = 0, 1,\ldots, D,\quad \text{let }\mathrm{Avg}_d : \mathbb{V}_d \longrightarrow \overline{\mathbb{V}}, M_d \longrightarrow \mathrm{Avg}_d(M_d)   \\
    &\text{such that for all } \vj = (j_0, j_1, \ldots, j_D) \in \overline{\vJ},\\
    &\hspace{10mm}(\mathrm{Avg}_d(M_d))_\vj := \begin{cases}
        \frac{1}{2}(M_d)_{\vj}, & \text{ for }j_d = 1,\\
        \frac{1}{2}\left( (M_d)_{\vj - \ve_d } + (M_d)_{\vj }\right), & \text{ for }j_d = 2,3, \ldots, n_d-1,  \\
        \frac{1}{2}(M_d)_{\vj-\ve_d}, & \text{ for }j_d = n_d,
    \end{cases} \\
    &\text{and let }\mathrm{Avg}: \mathbb{V}_1 \times\ldots \times \mathbb{V}_D \rightarrow \overline{\mathbb{V}}^D, \vM \rightarrow (\mathrm{Avg}_1(M_1), \ldots, \mathrm{Avg}_D(M_D)),
\end{aligned}
\end{equation}
{where the averaged values at $j_d = 1$ and $j_d = n_d$ are 
$\frac{1}{2}(M_d)_{\vj}$ and $\frac{1}{2}(M_d)_{\vj-\ve_d}$, respectively 
{because of} the boundary condition $\vm \cdot \vn = 0$ in Problem \eqref{problem: main}.}  The adjoint operators of the average operators defined in~\eqref{eq:Avgop} are given by
\begin{equation}
\begin{aligned}
&\text{for each } d = 0, 1,\ldots, D,\quad \mathrm{Avg}^*_d : \overline{\mathbb{V}} \longrightarrow \mathbb{V}_d, \widehat{M}_d \longrightarrow \mathrm{Avg}^*_d(\widehat{M}_d),\\
&\text{such that for all } \vj = (j_0, j_1, \ldots, j_D) \in \widehat{\vJ}\\
&\hspace{1cm}\mathrm{Avg}^*_d(\widehat{M}_d)_\vj := \frac{1}{2}\left( (\widehat{M}_d)_{\vj} + (\widehat{M}_d)_{\vj+\ve_d} \right)\\
&\text{and }\mathrm{Avg}^*: \overline{\mathbb{V}}^D  \rightarrow \mathbb{V}_1 \times\ldots \times \mathbb{V}_D, \widehat{\vM} \rightarrow (\mathrm{Avg}^*_1(\widehat{M}_1), \ldots, \mathrm{Avg}^*_D(\widehat{M}_D)).
\end{aligned}
\end{equation}

Also, we let $b_{\mathrm{avg}} \in \overline{\mathbb{V}}$ contain density boundary values: for any $\vj = (j_0, j_1, \ldots, j_D) \in \overline{\vJ}$,
\begin{align}\label{eq:bavgdisc}
    (b_{\mathrm{avg}})_\vj:= \begin{cases}
        \frac{1}{2}\bm{\rho}_0(\vx_{\hat{\vj}}), & \text{ for }j_0 = 1,\\
        0, & \text{ for }j_0 = 2, 3,\ldots, n_0-1\\
        \frac{1}{2}\bm{\rho}_1(\vx_{\hat{\vj}}), & \text{ for } j_0 = n_0,
    \end{cases}
\end{align}
where {$\hat{\vj} = (j_1, \ldots, j_D)
$ takes the last $D$ entries of $\vj$.}
Using the above operators, 
we define $\widetilde{M_0}, \widetilde{\vM} \in \overline{\mathbb{V}}$ as 
\begin{align}
    \widetilde{M_0} := \mathrm{Avg}_0(M_0) + b_{\mathrm{avg}}, \quad  \widetilde{\vM} := \mathrm{Avg}(\vM),
\end{align}
and approximate the objective value of \eqref{problem: main} by
\begin{align}\label{eq:discobj}
    f(\widetilde{M_0}, \widetilde{\vM}) := \left( \prod_{d=0}^D \Delta_d \right) \mathlarger{\mathlarger{\sum_{\vj \in \overline{\vJ}}}} \frac{\|\widetilde{\vM}_\vj \|^2}{2(\widetilde{M_0})_\vj}.
\end{align}


\paragraph{Constraints:}
 We use finite difference method to discretize the PDE constraints in \eqref{problem: main}. We first define  the following operators for discretizing partial derivatives:
 \begin{equation}
\begin{aligned}\label{eq:parOp}
    &{\text{for each } d = 0, 1,\ldots, D,}\quad \text{let }\cP_d: \mathbb{V}_d \rightarrow \overline{\mathbb{V}}, M_d \rightarrow \cP_d(M_d), 
    \\
    &{\text{such that for all } \vj = (j_0, j_1, \ldots, j_D) \in \overline{\vJ}},\\
    &\hspace{1cm}(\cP_d(M_d))_{\vj} := \begin{cases}
        \frac{1}{\Delta_d}(M_d)_{\vj}, & \text{ for }j_d = 1\\
        \frac{1}{\Delta_d} \left( (M_d)_{\vj} - (M_d)_{\vj-\ve_d} \right), & \text{ for }j_d =  2, 3,\ldots n_d - 1 \\
        -\frac{1}{\Delta_d}(M_d)_{\vj-\ve_d}, & \text{ for }j_d = n_d,
    \end{cases},\\
    &\text{and let }\cP: \mathbb{V}_1 \times\ldots \times \mathbb{V}_D \rightarrow \overline{\mathbb{V}}^D, \vM \rightarrow \big(\cP_1(M_1), \ldots, \cP_D(M_D)\big).
\end{aligned}
\end{equation}
{Just like in the definition of the average operators in \eqref{eq:Avgop}, the boundary values of $\vM$ are already included in \eqref{eq:parOp}.} {The adjoint operators of the partial derivative operators in \eqref{eq:parOp} are given by 
\begin{equation}
\begin{aligned}\label{eq:parOpAdj}
    &\text{for each } d = 0, 1,\ldots, D,\quad
    \cP_d^*: \overline{\mathbb{V}} \rightarrow \mathbb{V}_d,\;
    \widehat{M}_d \rightarrow \cP_d^*(\widehat{M}_d),\\
    &\text{such that for all } \vj = (j_0, j_1, \ldots, j_D) \in \widehat{\vJ}_d,\\
    &\hspace{1cm}(\cP_d^*(\widehat{M}_d))_{\vj}
    :=
    \frac{1}{\Delta_d}\left( (\widehat{M}_d)_{\vj} - (\widehat{M}_d)_{\vj+\ve_d}\right),\\
    &\text{and }\cP^*: \overline{\mathbb{V}}^D \rightarrow \mathbb{V}_1 \times\ldots \times \mathbb{V}_D,\;
    \vP \rightarrow \big(\cP_1^*(P_1), \ldots, \cP_D^*(P_D)\big).
\end{aligned}
\end{equation}
}

Second, we let 
$b \in \overline{\mathbb{V}}$ contain the $M_0$ boundary conditions: for all $\vj = (j_0, j_1, \ldots, j_D) \in \overline{\vJ}$,
\begin{align}\label{eq:bdisc}
    b_\vj := \begin{cases}
        -\frac{1}{\Delta_0}\bm{\rho}_0(\vx_{\hat{\vj}}), & \text{ for }j_0 = 1,\\
        0, & \text{ for }j_0 = 2, 3,\ldots, n_0-1\\
        \frac{1}{\Delta_0}\bm{\rho}_1(\vx_{\hat{\vj}}), & \text{ for }j_0 = n_0,
    \end{cases}
\end{align}
{where 
{$\hat{\vj} = (j_1, \ldots, j_D)
$ takes the last $D$ entries of $\vj$.}} 
With these notations, the PDE constraint set in Problem \eqref{problem: main} can be approximated by 
\begin{align}\label{def:const_set}
    \mathcal{C} := \left\{(M_0, \vM)\,\Bigg|\,  \sum_{i=0}^D \cP_dM_d + b = \mathbf{0} \right\}.
\end{align}

Now we obtain the discretized version of Problem~\eqref{problem: main} as follows:
{
\begin{equation}
    \begin{aligned}\label{prob:pre_admm}
        \min_{(M_0, \vM) \in \mathcal{C}} f\left(\mathrm{Avg}_0(M_0) + b_{\mathrm{avg}}, \mathrm{Avg}(\vM)\right),
    \end{aligned}
\end{equation}
where $M_0$ and $\vM$, defined in \eqref{eq:discrePM}, are the discretized counterparts of $\rho$ and $\vm$, respectively. The function $f$ is defined in \eqref{eq:discobj}, the averaging operators $\mathrm{Avg}_0$ and $\mathrm{Avg}$ are given in \eqref{eq:Avgop}, $b_{\mathrm{avg}}$ is defined in \eqref{eq:bavgdisc}, and the set $\mathcal{C}$ is defined in \eqref{def:const_set}.
}


\subsection{Proximal Mapping \texorpdfstring{{of $F$}}{of F}}
\label{subsec: proxmap}
    Note that the cost function $F$ defined in \eqref{eq: maincostfunction} is differentiable. However, its gradient is not (globally) Lipschitz continuous. Even locally, its gradient Lipschitz constant can be extremely large, in particular when $b$ is close to $0$. 
When $b>0$, we have the Hessian matrix of $F$ as 
        $\nabla^2 {F(b, \va)} 
        = \begin{bmatrix}
            \frac{\|\va\|^2}{b^3}  & \frac{\va^\top}{b^2} \\[2mm]
             -\frac{\va}{b^2} & \frac{1}{b}I
        \end{bmatrix}.$ 
    Thus 
      $\|\nabla^2 F(b, \va)\|  = \Theta\left( \frac{\|\va\|^2}{b^3} \right)$.
    Since the Lipschitz constant of {$\nabla F$} for $b>0$ is roughly the norm of its Hessian, 
    it can become arbitrarily large as $b$ approaches zero. Consequently, if either $\bm{\rho}_0$ or $\bm{\rho}_1$ in \eqref{problem: main} takes values close to zero, the (local) gradient Lipschitz constant becomes large, leading to tiny step sizes in the update of gradient-type methods for guaranteed convergence. 

 To address the issue caused by the exploding gradient Lipschitz constant, we do not perform gradient update by using $\nabla F$ but instead use its 
proximal mapping which is provided by the lemma below. We defer its proof in the appendix. 

\begin{lemma}\label{lemma:proxmain}
    {The proximal mapping of function $F$ \eqref{eq: maincostfunction} with parameter $\gamma > 0$ is 
    \begin{equation}
    \begin{aligned}
        \mathrm{Prox}_{F, \gamma}(\hat{\rho}, \hat{\vm}) &:= \argmin_{\rho, \vm} F(\rho, \vm) + \frac{1}{2\gamma} \left( (\rho - \hat{\rho})^2 + \|\vm-\hat{\vm}\|^2 \right), \\
        &=
        \begin{cases}
(0,\mathbf{0}),
& \text{if } \widehat{\rho} \le 0,\; \| \widehat{\vm}\|^2 \le -2\gamma \widehat{\rho},
\\[6pt]
\left(\overline{\rho}, \dfrac{\widehat{\vm}\,\overline{\rho}}{\overline{\rho}+\gamma} \right),
& \text{otherwise}.
\end{cases}
\end{aligned}
\end{equation}
    }
Here 
$\overline{\rho}$ is the unique real-number solution of the cubic equation 
\[
\overline{\rho}^3
+ (2\gamma - \widehat{\rho})\overline{\rho}^2
+ (\gamma^2 - 2\gamma \widehat{\rho})\overline{\rho}
- \frac{1}{2}\gamma \|\widehat{\vm}\|^2
- \gamma^2 \widehat{\rho}
= 0.
\]
\end{lemma}

\begin{remark}
    The proximal mapping of the function \(F\), defined in Lemma~\ref{lemma:proxmain}, plays a central role in the {superior} performance of our proposed LADMM. 
    In contrast, many existing methods for solving Problem~\eqref{problem: main}, including ALM~\cite{benamou2000computational}, FISTA~\cite{yu2024fast},  and G-prox~\cite{jacobs2019solving}, do not explicitly utilize this proximal mapping. 
\end{remark}


\subsection{Linearized ADMM}\label{subsec:ladmm}

In this subsection, we derive LADMM updates to solve {Problem \eqref{prob:pre_admm}.}
{
 We employ a variable-splitting technique by introducing auxiliary variables $\overline{\vM}$ and $\overline{M_0}$ corresponding to the averaged quantities of $\vM$ and $M_0$, respectively. These auxiliary variables enable 
 the use of the exact proximal mapping described in Lemma \ref{lemma:proxmain}. With these variables, Problem \eqref{prob:pre_admm} can be reformulated as
}

\begin{equation}
\begin{aligned}\label{prob:ADMM1ag}
     \min_{\overline{M_0}, \overline{\vM}, (M_0, \vM) \in \mathcal{C}}& f(\overline{M_0},
\overline{\vM})\\
&\st \mathrm{Avg}(\vM) - \overline{\vM} = \mathbf{0}, \mathrm{Avg}_0(M_0) + b_{\mathrm{avg}} - \overline{M_0} = 0.
\end{aligned}
\end{equation}
Note that we intentionally retain the mass-conservation constraint \((M_0,\vM)\in\mathcal{C}\) in the formulation and will 
preserve it 
at each iteration of our algorithm by performing an exact projection. 

The AL function of \eqref{prob:ADMM1ag} is defined as 
\begin{align}\label{eq: aug1ag}
\mathcal{L}&_\beta(M_0, \vM, \overline{M_0}, \overline{\vM}; \mathbf{\Pi}, \Lambda)
=  f(\overline{M_0}, \overline{\vM})
   + \langle \mathbf{\Pi}, \mathrm{Avg}(\vM) - \overline{\vM} \rangle_{\Delta} \notag \\
   &+   \frac{ \beta}{2}\|\mathrm{Avg}(\vM) - \overline{\vM}\|_{\Delta}^2 
      + \langle \Lambda, \mathrm{Avg}_0(M_0) + b_{\mathrm{avg}} - \overline{M_0} \rangle_{\Delta} \notag \\
    &  +  \frac{ \beta}{2}\|\mathrm{Avg}_0(M_0) + b_{\mathrm{avg}} - \overline{M_0}\|_{\Delta}^2,
\end{align}
where $\Lambda$ and $\mathbf{\Pi}$ are the dual variables corresponding to the averaged $M_0$ and $\vM$ constraints respectively, and $\Delta = \prod_{d=0}^D \Delta_d $. 
With the AL function, {the LADMM iterative updates} can be written as follows: for each $k\ge0$,
\begin{align}
(\overline{M_0}^{k+1}, \overline{\vM}^{k+1})
&= {\argmin_{\substack{(\overline{M_0},\,\overline{\vM})}}
\cL_{\beta_k}(M_0, \vM, \overline{M}_0^k, \overline{\vM}^k; \mathbf{\Pi}^k, \Lambda^k)},
\label{eq:update_bar}
\\[2pt]
(M_0^{k+1}, \vM^{k+1}) &= \argmin_{(M_0, \vM) \in \mathcal{C}}
    \frac{1}{2}\,\lVert \vM - \vM^{k+\frac{1}{2}} \rVert^{2} + \frac{1}{2}\,\lVert M_0 - M_0^{k+\frac{1}{{2}}} \rVert^{2}, 
\label{eq:update_PM}
\\[2pt]
(\mathbf{\Pi}^{k+1}, \Lambda^{k+1})
&= (\mathbf{\Pi}^k, \Lambda^k)
+ \beta_k\Bigl(
\mathrm{Avg}(\vM^{k+1}) - \overline{\vM}^{k+1},\;
 \notag \\
&\qquad\qquad\qquad\qquad \mathrm{Avg}_0(M_0^{k+1}) + b_{\mathrm{avg}} - \overline{M_0}^{k+1}
\Bigr).\label{eq:dualupdate}
\end{align}
Here, we use LADMM~\cite{lin2011linearized, chan2012linearized, ouyang2015accelerated, lin2017extragradient, yang2015fast} instead of standard ADMM to update \((M_0^{k+1},\vM^{k+1})\), since the corresponding ADMM subproblem does not admit a closed-form solution.
More precisely, \eqref{eq:update_PM} is obtained by linearizing $\cL_{\beta_k}(\cdot, \cdot, \overline{M}_0^{k+1}, \overline{\vM}^{k+1}; \Pi^k, \Lambda^k)$ at $(M_0^k, \vM^k)$ and adding a proximal term. Since the terms involving $M_0$ and $\vM$ are independent and the gradient Lipschitz constants satisfy $\|\mathrm{Avg}_0^*\mathrm{Avg}_0\|_*
=
\cos\!\left(\frac{\pi}{2n_0}\right) \le 1$ and $\|\mathrm{Avg}^*\mathrm{Avg}\|_* 
=
\prod_{i=1}^{D}\cos\!\left(\frac{\pi}{2n_i}\right) \le 1,$
 we 
set
\begin{equation}\label{eq: intmPM_1ag}
\begin{aligned}
  &M_0^{k+\frac{1}{2}}
  = M_0^{k}
    - \frac{1}{\beta_k}
      \mathrm{Avg}_0^{*}\bigl(\Lambda^{k}
      + \beta_k\, (\mathrm{Avg}_0(M_0^k) + b_{\mathrm{avg}} - \overline{M_0}^{k+1}) \bigr)
    ,  \\
  &\vM^{k+\frac{1}{2}}
  = \vM^{k}
    - \frac{1}{\beta_k}
      \mathrm{Avg}^{*} \bigl(\mathbf{\Pi}^{k}
      + \beta_k\, \mathrm{Avg}(\vM^k) - \overline{\vM}^{k+1}\bigr).
\end{aligned}
\end{equation}

 The $(\overline{M}_0, \overline{\vM})$-update in \eqref{eq:update_bar} can be implemented by utilizing the proximal mapping in Lemma~\ref{lemma:proxmain}. 
%
The projection subproblem in \eqref{eq:update_PM} can be rewritten as 
\begin{equation}\label{eq:proj-subprob}
\begin{aligned}
    \min_{M_0,M}& \frac{1}{2}\|M_0-M_0^{k+ \frac12}\|^2 + \frac{1}{2}\|\vM-\vM^{k+\frac12}\|^2, \\
& \st  \cP_0 M_0 + \cP \vM + b = \mathbf{0}, \\
\end{aligned}
\end{equation}
{where $\cP_0$ and $ \cP$ are defined in \eqref{eq:parOp}.}
Let $\phi \in \overline{\mathbb{V}}$ be the multiplier to the linear constraint. 
To find the optimal solution of \eqref{eq:proj-subprob}, we form the KKT system as follows 
\begin{align} 
&M_0 - M_0^{k+\frac{1}{2}} + \cP_0^* \phi = \mathbf{0}, \label{eq:localPupdate}\\    
&\vM - \vM^{k + \frac{1}{2}} + \cP^* \phi = \mathbf{0}, \label{eq:localMupdate}\\
&\cP_0 M_0 + \cP \vM + b = \mathbf{0}. \label{eq:localcons}
\end{align}
 Substituting $M_0 = M_0^{k+\frac{1}{2}} - \cP_0^* \phi$ and $\vM = \vM^{k + \frac{1}{2}} - \cP^* \phi$ from equations \eqref{eq:localPupdate} and \eqref{eq:localMupdate}, respectively, to \eqref{eq:localcons} yields 
\begin{align}\label{eq:preSylv}
    \cP_0 \cP_0^* \phi + \cP \cP^* \phi = \cP_0 M_0^{k+\frac{1}{2}} + \cP \vM^{k + \frac{1}{2}} + b .
\end{align}
This is essentially a Poisson equation, which can be efficiently computed using the fast cosine transform, in a manner similar to the projection step in~\cite{yu2024fast}. 

Using discretization matrices to represent the operators $\cP_0$ and $\cP$ results in the 
equation
\begin{align}\label{eq:sylvgen}
      \sum_{d=0}^D \phi \times_d A_d = C:=\cP_0M_0^{k+\frac{1}{2}} + \cP \vM^{k+\frac{1}{2}} + b, 
\end{align}
where $\times_d$ represents the tensor multiplication~\cite{kolda2009tensor} in the $x_d$ direction for $d > 0$ and $\times_0$ represents the tensor multiplication in the $t$ direction, 
and $A_d$ is the matrix representation (that is given below for the case of $D=1$) of $\cP_d\cP_d^*$, for all $d = 0, 1, 2, \ldots, D$. An efficient way to solve \eqref{eq:sylvgen} is by using the eigendecomposition of matrices $A_d$ for all $d = 0, 1, \ldots, D$.

With the above derivations,  we 
give the pseudocode of our method in Algorithm \ref{alg:ladmm1ag}. 
\begin{algorithm}[htbp]
\caption{LADMM for dynamic OT}
\begin{algorithmic}[1]
\Require $\bm{\rho}_0,\bm{\rho}_1$
\State \textbf{Initialization:} 
$\vM_d^{0}(\vj) = 0 \text{ for } \vj \in \hat{\vJ}_d, \text{ for all } d = 0, 1, \ldots, D$. 

\For{$k=0,1,2,\ldots$}
\State Pick $\beta_k > 0$ and update $(\overline{M}_0, \overline{\vM})$ by
$$\begin{aligned}
\left(\overline{M_0}^{k+1}, \overline{\vM}^{k+1}\right)
= \mathrm{Prox}_{\left(F, \frac{1}{ \beta_k}\right)} \left(\mathrm{Avg}_0(M_0^k) + b_{\mathrm{avg}} - \frac{\Lambda^k}{ \beta_k}, \mathrm{Avg}(\vM^k) - \frac{\mathbf{\Pi}^k}{  \beta_k} \right). 
\end{aligned}$$

\State 
Let $(M_0^{k+\frac{1}{2}}, M^{k+\frac12})$ be defined in \eqref{eq: intmPM_1ag} and solve the following equation for $\phi^{k+1}$:
$$\begin{aligned}
\cP_0\cP_0^*\phi^{k+1} + \cP\cP^*\phi^{k+1} = \cP_0M_0^{k+\frac{1}{2}} + \cP\vM^{k+\frac12}+b. 
\end{aligned}$$
\State  Update $(M_0, \vM)$ by 
$$\begin{aligned}
M_0^{k+1} = M_0^{k+\frac12} - \cP_0^*\phi^{k+1},\
\vM^{k+1} &= \vM^{k+\frac12} - \cP^*\phi^{k+1}.
\end{aligned}$$

\State Update multipliers by 
$$\begin{aligned}
 \mathbf{\Pi}^{k+1}
&= \mathbf{\Pi}^k
+ \beta_k\left( \mathrm{Avg}(\vM^{k+1}) - \overline{\vM}^{k+1}\right), \\
\Lambda^{k+1} &= \Lambda^k + \beta_k\left(\mathrm{Avg}_0(M_0^{k+1}) + b_{\mathrm{avg}} - \overline{M_0}^{k+1}\right).
\end{aligned}$$
\EndFor
\end{algorithmic}
\label{alg:ladmm1ag}
\end{algorithm}

\begin{remark}(Explicit solution of the equation in \eqref{eq:sylvgen})
To explain how to solve \eqref{eq:sylvgen}, we start from the case of $D=1$. In this case, \eqref{eq:sylvgen} becomes the Sylvester equation
\begin{align}\label{eq: AB}
    A_0 \phi + \phi A_1 = C, 
\end{align}
where $A_0$ and $A_1$ are tri-diagonal positive semi-definite matrices: 
\begin{align}
    A_0 = \frac{1}{\Delta_0^{2}}
\begin{bmatrix}
1      & -1     & 0      & \cdots & 0 \\
-1     & 2      & -1     & \ddots & \vdots \\
0      & -1     & 2      & \ddots & 0 \\
\vdots & \ddots & \ddots & \ddots & -1 \\
0      & \cdots & 0      & -1     & 1
\end{bmatrix}_{n_0 \times n_0},
    \quad A_1 = \frac{1}{\Delta_1^{2}}
\begin{bmatrix}
1      & -1     & 0      & \cdots & 0 \\
-1     & 2      & -1     & \ddots & \vdots \\
0      & -1     & 2      & \ddots & 0 \\
\vdots & \ddots & \ddots & \ddots & -1 \\
0      & \cdots & 0      & -1     & 1
\end{bmatrix}_{n_1 \times n_1}, 
\end{align}
with $\Delta_0 = \frac{1}{n_0}$ and $\Delta_1 = \frac{1}{n_1}$ representing the grid spacings along the $t$- and $x_1$- directions, respectively. Let $A_0 = U \Sigma_0 U^\top$ and $A_1 = V \Sigma_1 V^\top$ be their eigen-decompositions.
Then 
$$\Sigma_0 U^\top \phi V + U^\top \phi V \Sigma_1 = U^\top C V,$$
which indicates (in an entry-wise manner)
$$(\Sigma_{0})_i (U^\top \phi V)_{i,j} + (U^\top \phi V)_{i,j} (\Sigma_1)_j = (U^\top C V)_{i,j}, \forall\, (i,j).$$
Hence, 
\begin{align}\label{eq: mainsylv}
    (U^\top \phi V)_{i,j} = \frac{(U^\top C V)_{i,j}}{(\Sigma_0)_i + (\Sigma_1)_j}, \forall\, (i,j).
\end{align}
We can obtain $U^\top \phi V$ via \eqref{eq: mainsylv} by computing the eigenvalues and eigenvectors of 
$A_0$ and $A_1$ in \eqref{eq: AB} once at the beginning of the algorithm and reusing them {for each update}. The matrix $A_0$ (similar for $A_1$) is well-studied in literature, and its eigenvalues and egivenvectors 
are given by 
\begin{equation}
\begin{aligned}\label{eq: eigval}
    &(\Sigma_0)_i = 4 n_0^2\sin^2\left(\frac{(i-1)\pi}{2n_0} \right), \forall i \in [n_0], \\
    &U^{(1)}_j = \frac{1}{\sqrt{n_0}}, U^{(i)}_j =  \sqrt{\frac{2}{n_0}} \cos \left(\left(j - \frac12 \right)\frac{(i-1)\pi}{n_0} \right), \forall j \in [n_0], \forall i \in [n_0] \setminus \{1\}.
    \end{aligned}
\end{equation}
Here $U^{(i)}$ is the $i$-th eigenvector corresponding to eigenvalue $(\Sigma_0)_i$ and the set $\{U^{(i)} \}_{i \in [n_0]}$ forms an orthonormal basis. {Note that $\Sigma_1$ and $V$ can be obtained by replacing $n_0$ with $n_1$ in \eqref{eq: eigval}.} Therefore, we can recover $\phi$ from $U^\top \phi V$ 
efficiently.

In a general case of $D\ge1$, we let $A_d = U_d \Sigma_d U_d^\top$ be the eigendecomposition for each $d=0,1,\ldots, D$. Then in analogy to \eqref{eq: mainsylv}, we have from \eqref{eq:sylvgen} that for any index $(i_0, i_1, \ldots, i_D)$,
\begin{align}\label{eq:sylv-general-D}
\begin{aligned}
 &   \left( \phi \times_0 U_0^\top \times_1 U_1^\top \times_2 \cdots \times_D U_D^\top \right)_{i_0, i_1, \ldots, i_D} \\
 = & \frac{\left( C \times_0 U_0^\top \times_1 U_1^\top \times_2 \cdots \times_D U_D^\top \right)_{i_0, i_1, \ldots, i_D}}{(\Sigma_0)_{i_0} + (\Sigma_1)_{i_1} + \cdots + (\Sigma_D)_{i_D}  }.
\end{aligned}    
\end{align}
With $\phi \times_0 U_0^\top \times_1 U_1^\top \times_2 \cdots \times_D U_D^\top$, we can easily obtain $\phi$ by multiplying $U_d$ to the $d$-th mode for all $d=0,1,\ldots,D$.
\end{remark}


The 
convergence result of LADMM has been established in the literature. We have the following theorem directly from Theorem 4.3 of~\cite{gao2017first} or Theorem 2.13 of \cite{xu2017accelerated}.
\begin{theorem}\label{The:conver}
Let $\big\{\big(\overline{M}_0^k, \overline{\vM}^k, M_0^k, \vM^k\big) \big\}_{k\ge0}$ be the sequence generated from Algorithm~\ref{alg:ladmm1ag} with $\beta_k=\beta, \forall\, k\ge0$ for some $\beta>0$. 
Define $$\big(\overline{M}_0^{\mathrm{erg},K}, \overline{\vM}^{\mathrm{erg},K}, M_0^{\mathrm{erg},K}, \vM^{\mathrm{erg},K}\big) = \frac{1}{K+1}\sum_{k=0}^K\big(\overline{M}_0^k, \overline{\vM}^k, M_0^k, \vM^k\big).$$
    Then $(M_0^{\mathrm{erg},K}, \vM^{\mathrm{erg},K}) \in \mathcal{C}$ and
    \begin{align*}
        \big |f(\overline{M}_0^{\mathrm{erg},K}, \overline{\vM}^{\mathrm{erg},K}) - f^* \big | + & \left\|\mathrm{Avg}_0(M_0^{\mathrm{erg},K}) + b_{\mathrm{avg}} - \overline{M}_0^{\mathrm{erg},K}\right\| \\
        &+ \left\|\mathrm{Avg}(\vM^{\mathrm{erg},K}) - \overline{\vM}^{\mathrm{erg},K}\right\|  = \mathcal{O}\left(\frac{1}{K} \right), 
    \end{align*}
    where $f^*$ denotes the optimal objective value of \eqref{prob:ADMM1ag}, and $\mathcal{C}$ is defined in \eqref{def:const_set}.
\end{theorem}

\section{Distributed implementation}
The variables $M_0$ and $\vM$ scale with 
the total number of grid points 
$
\prod_{i=0}^D n_i,
$
thus the storage requirement increases exponentially with the spatial dimension, $D$.
To address this challenge, 
we propose solving \eqref{prob:ADMM1ag} by using multiple agents, where one agent 
may correspond to a GPU, 
or a CPU core. 
The distributed 
computation not only alleviates memory constraints but also enables significant computational speedup through parallel execution. By distributing the workload across multiple agents, the method allows many computations to be carried out in parallel, 
thereby reducing overall runtime. This parallel structure improves scalability and makes the approach particularly well-suited for large-scale problems.  
Below we discuss how the subproblems appearing in Algorithm~\ref{alg:ladmm1ag} can be solved across multiple agents 
and detail the communication requirements of the algorithm.

\subsection{Distributed Implementation of \texorpdfstring{$(\overline{M}_0, \overline{\vM})$}{M} and Multiplier Updates}\label{subsec:local_imp}

Suppose that $N$ agents are available. We partition the variables $M_0$, $\vM$, $\overline{M}_0$, $\overline{\vM}$, $\mathbf{\Pi}$, and $\Lambda$ across the agents by decomposing the computational grid along the $x_1$-dimension, i.e., the first spatial dimension. {For simplicity of communication, we adopt a domain-decomposition strategy in this work. We note, however, that other decomposition strategies are also possible within the proposed framework}. It is illustrated in Figure~\ref{fig:distrPM} for the case {of} $D=1$. Under this domain decomposition, 
we describe how the primal and dual updates in \eqref{eq:update_bar}, \eqref{eq:update_PM}, and \eqref{eq:dualupdate} can be implemented in a distributed manner across multiple agents.
\begin{figure}[htbp]
    \centering
    \includegraphics[width=0.95\linewidth]{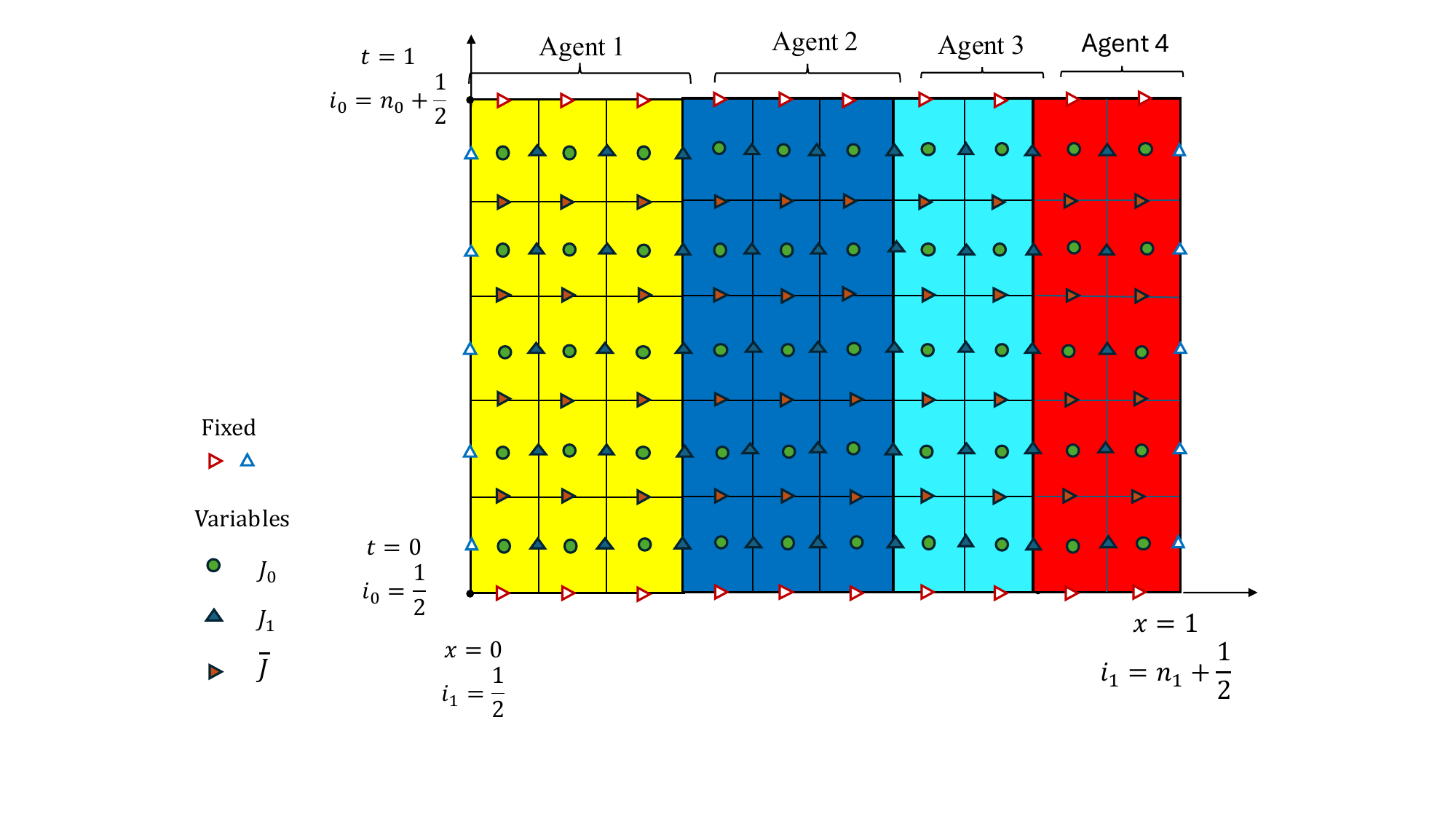}
    \caption{Illustration of the domain decomposition for the case $D=1$ {, $N=4, n_0 = 5, n_1 = 10$, where the domain is vertically decomposed along the spatial dimension.}  } 
    \label{fig:distrPM}
\end{figure}

According to our domain decomposition, each agent $i$ holds the following variables:
\begin{align}
    (M_0)_i, \vM_i, (\overline{M}_0)_i, \overline{\vM}_i, \mathbf{\Pi}_i, \Lambda_i.
\end{align}
{The variables on the shared boundary are always kept by the left agent.}
Hence these local variables do not overlap on the shared boundary and their concatenation will become exactly the entire variables $M_0$, $\vM$, $\overline{M}_0$, $\overline{\vM}$, $\mathbf{\Pi}$, and $\Lambda$.

To 
update $((\overline{M}_0)_i, \overline{\vM}_i)$ in \eqref{eq:update_bar} by agent $i$ locally, it suffices for agent $i$ 
to receive entries of $ \vM^{k}_{i-1}$ at the shared $x_1$-boundary between agents $i-1$ and $i$ (thus not the entire $ \vM^{k}_{i-1}$). 
Similarly, for the dual update to $(\mathbf{\Pi}_i, \Lambda_i)$ in \eqref{eq:dualupdate}, 
agent $i$ only needs to receive entries of $\vM^{k+1}_{i-1}$ at the shared $x_1$-boundary between agents $i-1$ and $i$.
The intermediate terms $M_0^{k+\frac{1}{2}}$ and $\vM^{k+\frac{1}{2}}$ defined in \eqref{eq: intmPM_1ag} for agent $i$ can be constructed by 
receiving entries of $\mathbf{\Pi}_{i-1}^{k}, \overline{\vM}_{i-1}^{k+1}, \vM_{i-1}^{k}$ from agent $i-1$ and entries of $\mathbf{\Pi}_{i+1}^{k}, \overline{\vM}_{i+1}^{k+1}, \vM_{i+1}^{k}$ from agent $i+1$, only at the shared $x_1$-boundary. 
This exchange of boundary information 
ensures that the operators $\mathrm{Avg}_0,  \mathrm{Avg}_0^*, \mathrm{Avg}, \text{ and }  \mathrm{Avg}^*$  are computed exactly.

\subsection{Distributed Implementation of \texorpdfstring{$({M}_0, {\vM})$}{M}-Updates}

{We next turn to the projection step in \eqref{eq:update_PM}, which requires solving \eqref{eq:sylvgen}. In contrast to the $(\overline{M}_0, \overline{\vM})$ and multiplier updates, this step involves more 
communication and 
requires a more careful treatment.}

\subsubsection{Distributed Implementation for Sylvester equation \texorpdfstring{\eqref{eq: AB}}{AB}}\label{sec:comm}

We first look at the case where $D=1$ in thorough detail, and the idea will follow similarly for $D>1$. 
When $D=1$, solving the projection problem \eqref{eq:update_PM} 
comes down to solving the Sylvester equation \eqref{eq: AB}, whose solution $\phi$ can be computed by using \eqref{eq: mainsylv}. 

To compute $\phi$ from \eqref{eq: mainsylv} in a distributed way, we partition $C$ and $\phi$ by the columns 
corresponding to the 
grid decomposition 
and partition $V$ by the rows; 
see Figure~\ref{fig:CUB} for illustration. Let $C_i$ and $\phi_i$ be the column blocks of $C$ and $\phi$ and $V_i$ be the row block of $V$ held by agent $i$ for each $i$. {In addition, each agent can do the multiplication by $U$ and $U^\top$ from the left. Notice that they do not need to store $U$ because the multiplication operation by $U$ (resp. $U^\top$) can be realized by a discrete cosine (resp. inverse discrete cosine) transformation.}

\begin{figure}[htbp] 
\centering
    \includegraphics[width=0.9\linewidth]{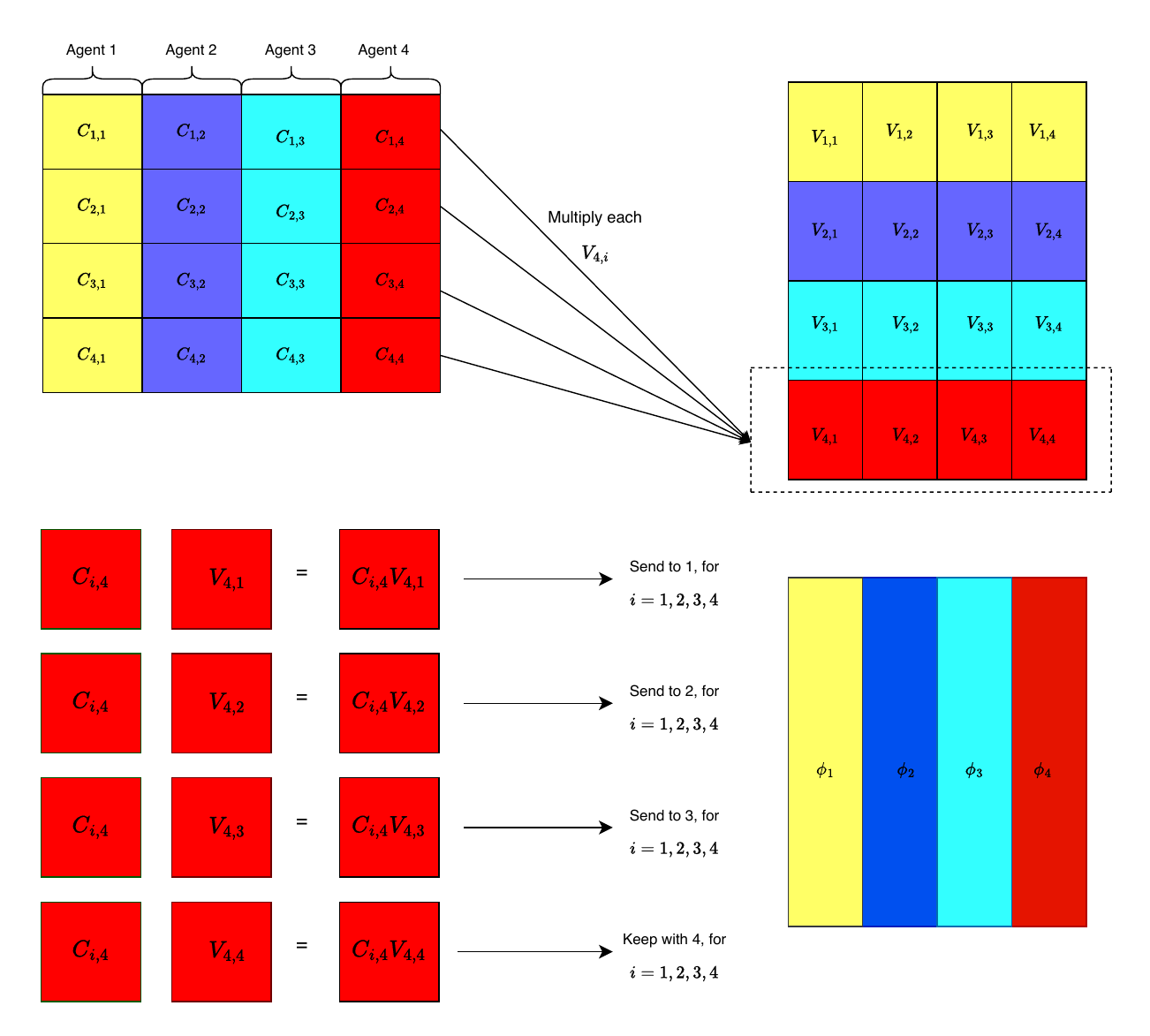}
\caption{Illustration of partition of  $C$, $V$, and $\phi$ for $N=4$, and communication using blocks of $C$ and $V$ at agent $4$. {Partition of $C$ and $V$ is depicted in the first row, partition of $\phi$ is depicted in the second column of the second row, and communication using the last column block of $C$ and last row block of $V$ is depicted in the first column of the second row.}}
\label{fig:CUB}
\end{figure}

Below, we explain in detail how to implement the multiplication $CV$ in a distributed way with the partition of $C$ and $V$ described above, such that each agent $i$ can obtain the $i$-th column block, denoted as $(CV)_i$, of $CV$. Let $C_{i,k}$ and $V_{i,k}$ be the $i$-th row and $k$-th column block of $C$ and $V$ respectively. Then the $i$-th row and $j$-th column block of $CV$ will be  
    $(CV)_{i, j} = \sum_{k=1}^N C_{i, k} V_{{k,j}}.$ 
Since $C_{i, k}$ and $V_{k,j}$ are held by agent $k$, to obtain $(CV)_{i, j}$ by agent $j$, it needs to receive the product $C_{i, k} V_{{k,j}}$ from all agents $k\neq j$ and then to perform the summation. Figure~\ref{fig:CUB} illustrates the case of $N=4$ and the message sent from agent 4 to all other agents.

{After forming the column block $(CV)_i$, the agent $i$ can obtain $U^\top (CV)_i$ by performing an inverse discrete cosine transformation to have the $i$-th column block of $U^\top C V$, denoted by $(U^\top C V)_i$}. Then we can easily obtain the $i$-th column block of the matrix $R$, whose $(a,b)$-th entry, denoted as $R[a,b]$, is $\frac{(U^\top C V)[a,b]}{(\Sigma_0)_a + (\Sigma_1)_b}$. From \eqref{eq: mainsylv}, it follows $U^\top \phi V = R$ and thus $\phi = U R V^\top$. To obtain $\phi$ in a distributed way, we use the same way of distributed implementation as we do for computing $U^\top C V$. This is illustrated in Figure~\ref{fig:RV}, and we do not repeat the details. With the $i$-th column block $\phi$ (denoted by $\phi_i$), as shown in Figure~\ref{fig:CUB}, agent $i$ can easily perform the update $(M_0^{k+1})_i = (M_0^{k+\frac{1}{2}})_i - \cP_0^* \phi_i$ locally without any more communication and $\vM^{k+1}_i=(\vM^{k + \frac{1}{2}})_i - \cP^* \phi_i$ {by receiving entries of $\phi_{i+1}$ from agent $i+1$ at the shared boundary.}

\begin{figure}[htbp] 
\centering
    \includegraphics[width=\linewidth]{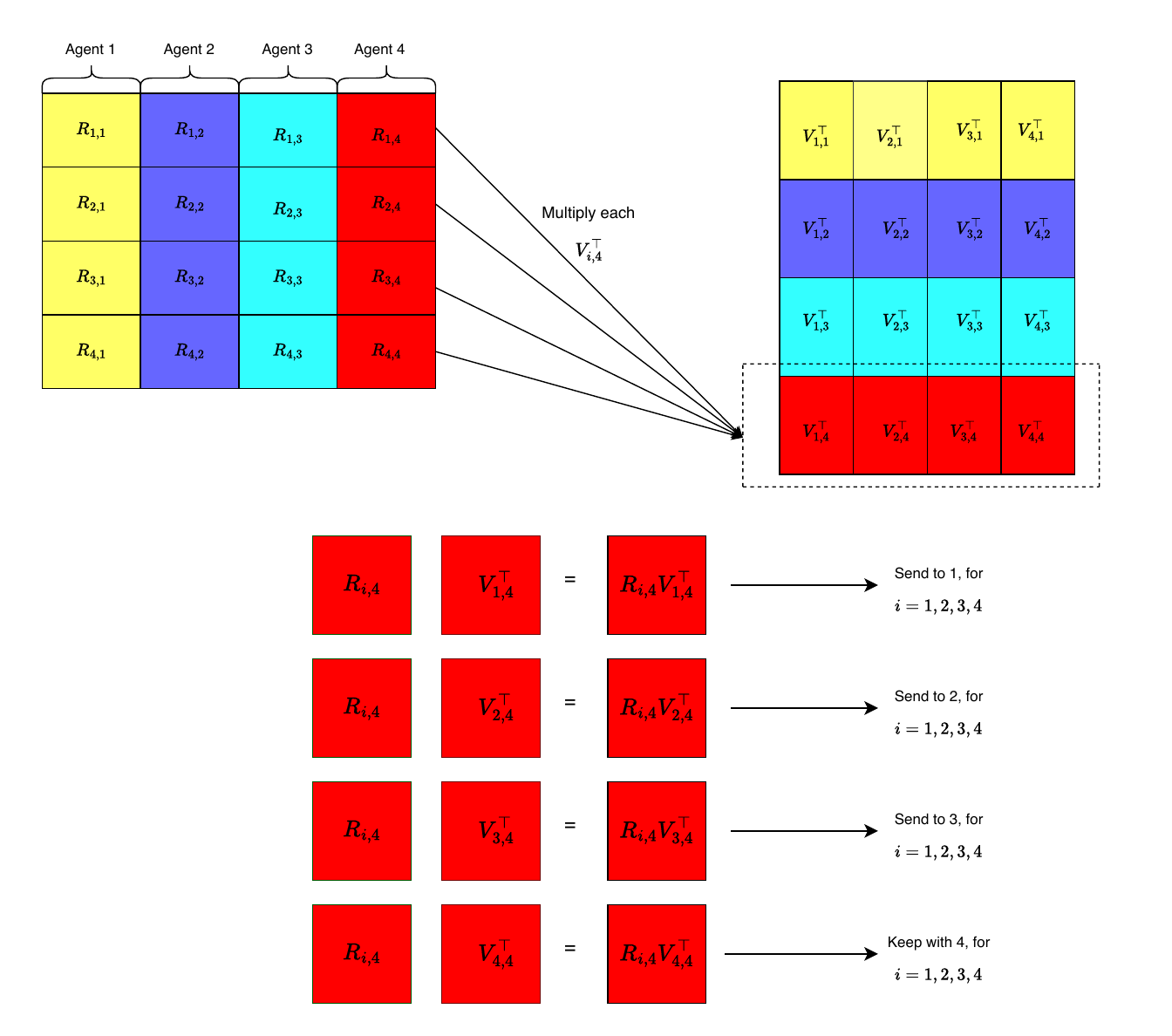}
\caption{Illustration of computation and communication using blocks of $R$ and $V^\top$ at agent 4 for $N=4$. {Partition of $R$ and $V^\top$ is depicted in the first row and communication using the last column block of $R$ and last row block of $V^\top$ is depicted in the second row.}}
\label{fig:RV}
\end{figure}

\subsubsection{Generalization to \texorpdfstring{$D>1$}{D>1}}

{When $D>1$, both $C$ and $\phi$ are $(D+1)$-way tensors. We still partition $C$ along the first space dimension $x_1$, as illustrated in the left picture of Figure~\ref{fig:CUB3D}, and we shard $\phi$ in the same manner. In analogy to the case of $D=1$ described in Section~\ref{sec:comm}, we partition $U_1$ by rows, and for all $d\neq 1$, each agent $i$ can perform the multiplication of $U_d$ to the $i$-th block of $C$ along the $d$-th mode by the discrete cosine transformation. The idea is the same as that for the case of $D=1$ and illustrated in the right picture of Figure~\ref{fig:CUB3D} for the case of $D=2$. We do not repeat it here.}

\begin{figure}[htbp] 
\centering
    \includegraphics[width=\linewidth]{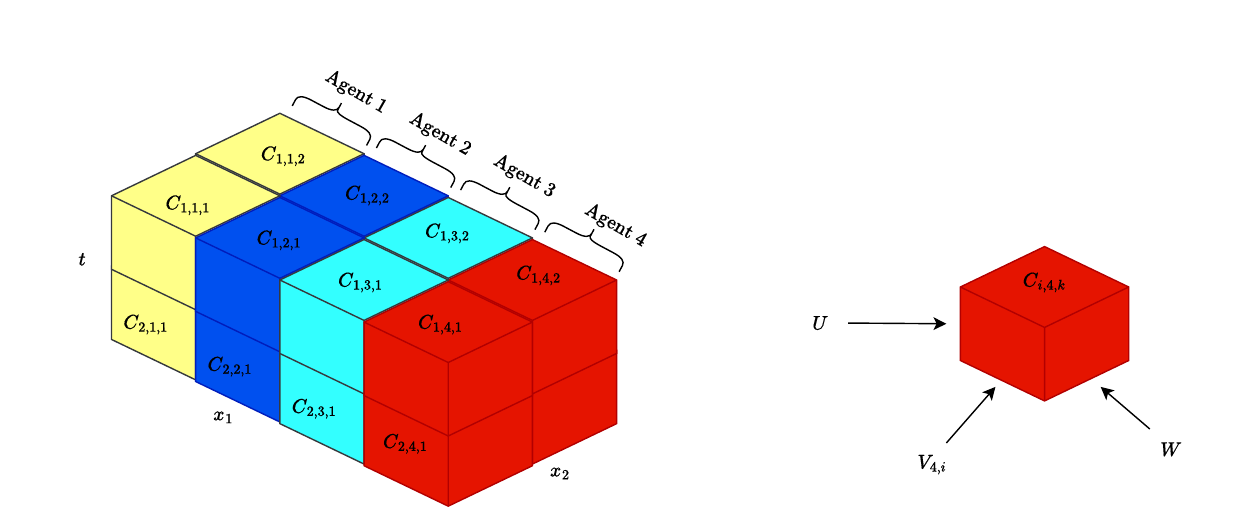}
\caption{Left: illustration of partition of $C$ along the first spatial dimension for $N = 4$ and $D=2$. 
Right: illustration of how the eige-matrices $U$, $V$, and $W$ act on a block of $C$.}
\label{fig:CUB3D}
\end{figure}

\section{Numerical Experiments}\label{sec:numexp}
We present numerical experiments for solving \eqref{prob:pre_admm}. In the nondistributed setting, we compare the proposed method, LADMM, with several existing approaches, including FISTA \cite{yu2024fast}, G-prox \cite{jacobs2019solving}, and ALM \cite{benamou1999numerical, benamou2000computational}. 
For brevity, we denote these methods by ``F'', ``G'', and ``A'', respectively, and use ``L'' to refer to the proposed LADMM. 
We also test the proposed method in a distributed setting.
All experiments are performed either in one- (i.e., $D=1$) or two-dimensional (i.e., $D=2$) spatial settings.

\subsection{Parameter setting}
For the proposed LADMM, we only need to tune the penalty parameter, and we choose it by 
\begin{align}
{\beta_k=\beta_0, \forall\, k \le 1000},\quad    \beta_k = \beta_0 k, \forall\, k > 1000,
\end{align}
where $\beta_0$ is selected via a grid search over the set $\{10^{-5}, 10^{-4}, 10^{-3}, 10^{-2}, 10^{-1},$ $ 1, 10\}$.  
{We update $\beta$ linearly after $1000$ iterations. This choice is motivated by numerical observations 
that the dual residual decreases more rapidly than the primal residual {in the beginning}. The linear update strategy balances this decreasing rate, resulting in improved overall convergence behavior.}

{FISTA~\cite{yu2024fast} performs projected gradient updates by line search. Following the setting in its code, we set its initial estimate of gradient Lipschitz constant to $L_0=10$ and the increase rate $\eta=1.2$ for the line-search of each update} {with an upper bound $L_{\max} = 10^5$.} 

{ALM~\cite{benamou2000computational} reformulates the objective in \eqref{problem: main} using a Legendre transform and solves its dual problem.} The penalty parameter \(r\) of ALM~\cite{benamou2000computational} is selected via grid search over the set $\{10^{-3}, 10^{-2},$ $ 10^{-1}, 1, 10, 10^2, 10^3\}$ for all numerical experiments. For G-prox~\cite{jacobs2019solving}, its primal proximal parameter \(\tau_p\) is chosen through grid search over \(\left\{ 10^{-5}, 10^{-4}, 10^{-3}, 10^{-2}, 10^{-1}, 1\right\}\), while the dual proximal parameter is set to \(\tau_d=\frac{1}{\tau_p}\). 
We observe that G-prox with 
$\tau_p > 1$ can diverge 
in some experiments. 
This behavior is likely attributable to the fact that the corresponding subproblem involves computing the proximal mapping of the function
\begin{align}
h(M_0, \vM) := \frac{\|\mathrm{Avg}(\vM)\|^2}{2\bigl(\mathrm{Avg}_0(M_0) + b_{\mathrm{avg}}\bigr)},
\end{align}
which does not admit a closed-form solution. {Hence}, the proximal mapping is obtained only approximately. Consequently, we restrict our grid search to values of $\tau_p \leq 1$. 

For all experiments, we initialize the primal and dual variables with the same settings: 
\begin{align}
    \vM^0 = \mathbf{0}, \quad M_0^0 = 0, \quad \Lambda^0 = 0, \quad \mathbf{\Pi}^0 = \mathbf{0}, \quad \phi^0 = 0, 
\end{align}
where $\phi$ is the dual variable for ALM and G-prox. The stopping condition { for ALM~\cite{benamou2000computational}, G-prox~\cite{jacobs2019solving}, and FISTA~\cite{yu2024fast}} is given by
\begin{align}\label{eq:stopcond}
    \max\!\left( \mu_k \sum_{i=0}^D \|M_i^{k+1} - M_i^k\|,\,  \left\| \sum_{i=0}^D \cP_i M_i^{k+1} + b \right\| \right) \leq 10^{-4},
\end{align}
where 
\begin{align*}
    \mu_k &= \beta_k, \quad \text{for LADMM and ALM},\\
    \mu_k &= 1, \quad \text{for FISTA, and G-prox }.
\end{align*}
Along with the condition in~\eqref{eq:stopcond}, LADMM {further needs} the condition
\begin{align}\label{eq:stopcond2}
\max\!\left(
\|\mathrm{Avg}(\vM^k) - \overline{\vM}^k\|,
\;\|\mathrm{Avg}_0(M_0^k) + b_{\mathrm{avg}} - \overline{M}_0^k\|
\right)
\leq 10^{-4}.
\end{align}

The first quantity in the stopping condition~\eqref{eq:stopcond} measures the dual error bound and the second term measures the primal error.
In all tables reported in this section, we present the quantities \texttt{Iter}, \texttt{Time}(s), $\texttt{Obj}_{\mathrm{proj}}$, and  \texttt{Pres}. 
Here, \texttt{Iter} denotes the total number of iterations required for the algorithm to terminate, while \texttt{Time}(s) represents the total computational time reported in seconds.
The quantity $\texttt{Obj}_{\mathrm{proj}}$ denotes the objective value obtained after projecting the final iterate $(M_0^K, M^K)$ onto the constraint set $\mathcal{C}$ defined in~\eqref{cons: main}. 
$\texttt{Obj}_{\mathrm{proj}} \,(10^{-2})$ simply means that values are reported after scaling by $10^{-2}$ (e.g., $0.0556$ is shown as $5.56$). Finally, \texttt{Pres} denotes the primal residual, i.e., the second term on the left-hand side of~\eqref{eq:stopcond}. 

Throughout the numerical experiments, { $g_{\mu, \Sigma}(\cdot)$ denotes the Gaussian density function with mean $\mu$ and covariance matrix $\Sigma$.} 

\subsection{Numerical results in a nondistributed setting}\label{sec:numexp1ag}
In this {subsection}, we test the algorithms in a nondistributed setting. 
For the one-dimensional spatial setting, 
we consider three different types of $\bm{\rho}_0$ and $\bm{\rho}_1$ 
in Sections~\ref{subsec: unnormalizedsingleagent}, \ref{subsec: normalizedsingleagent},  and \ref{subsec:numer_set3_1ag}. 
For the two-dimensional spatial setting, the numerical comparison is presented in Section~\ref{sec:num2d1agent}.

\subsubsection{Unnormalized
\texorpdfstring{$\bm{\rho}_0$ and $\bm{\rho}_1$}
{rho0 and rho1}
with an Added Positive Constant}\label{subsec: unnormalizedsingleagent}
In the first experiment, we generate $\bm{\rho}_0$ and $\bm{\rho}_1$ by 
\begin{equation}\label{eq:rhosetting_notnormal}
\begin{aligned}
    &
    ({{\rho}}_0)_i = g_{1/3, 0.1}\left(\frac{i}{n_1}\right) + \kappa, \quad ({{\rho}}_1)_i = g_{2/3, 0.1}\left(\frac{i}{n_1}\right) + \kappa, 
    \text{ for all }   i = 0, 1, \ldots, n_1.
\end{aligned}
\end{equation}
The parameter \(\kappa\in\{0.1,0.01\}\) in \eqref{eq:rhosetting_notnormal} is introduced to prevent the density from becoming too close to zero, as is commonly done in existing methods.

\begin{table}[htbp]
\caption{Comparison of FISTA~\cite{yu2024fast}, ALM~\cite{benamou2000computational}, G-prox~\cite{jacobs2019solving}, and LADMM in 1D space with $\bm{\rho}_0, \bm{\rho}_1$ in \eqref{eq:rhosetting_notnormal}, $\kappa = 0.1, \text{ and the number of time steps is } 64$. The best performance is highlighted in bold.}
\label{tab:1deasycase1}
\centering
\scriptsize
\setlength{\tabcolsep}{1pt}
\resizebox{0.9\textwidth}{!}{
\begin{tabular}{|c|cccc|cccc|cccc|cccc|}
\hline
\multirow{2}{*}{$n_1$} & \multicolumn{4}{c|}{\texttt{Iter}} & \multicolumn{4}{c|}{\texttt{Time}(s)} & \multicolumn{4}{c|}{$\texttt{Obj}_{\text{proj}}(10^{-2})$} & \multicolumn{4}{c|}{\texttt{Pres}} \\
\cline{2-17}
 & A & F & G & L & A & F & G & L & A & F & G & L & A & F & G & L \\
\hline
256 & 1074 & 843 & 1986 & \textbf{490} & 10 & 14 & 14 & \textbf{5} & 5.24025 & 5.23701 & 5.23702 & \textbf{5.23700} & 2.5E-5 & \textbf{5.7E-13}  & 3.3E-6 & 6.1E-11 \\
512 & 2119 & 839 & 2228 & \textbf{569} & 26 & 27 & 30 & \textbf{8} & 5.24036 & 5.23714 & 5.23715 & \textbf{5.23714} & 7.8E-6 & \textbf{2.1E-12} & 7.8E-6 & 1.3E-10 \\
1024 & 3834 & \textbf{839} & 2434 & 849 & 103 & 46 & 54 & \textbf{28} & 5.24037 & 5.23718 & 5.23718 & \textbf{5.23717} & 9.9E-5 & \textbf{8.7E-12} & 1.9E-5 & 5.1E-10  \\
2048 & 7014 & 839 & 2910 & \textbf{636} & 378 & 90 & 139 & \textbf{44} & 5.24040 & 5.23719 & 5.23719 & \textbf{5.23718}  & 9.9E-5 & \textbf{4.3E-11} & 9.9E-5 &  5.0E-09 \\
4096 & 9300 & 838 & 5644 & \textbf{671} & 1160 & 262 & 620 & \textbf{82} & 5.48069 & 5.23719 & 5.23719 & \textbf{5.23718} & 1.5E-4 & \textbf{1.5E-10} & 9.9E-5 & 1.0E-09 \\
\hline
\end{tabular}}
\end{table}

\begin{table}[htbp]
\caption{Comparison of FISTA~\cite{yu2024fast}, ALM~\cite{benamou2000computational}, G-prox~\cite{jacobs2019solving}, and LADMM in 1D space with $\bm{\rho}_0, \bm{\rho}_1$ in \eqref{eq:rhosetting_notnormal}, $\kappa = 0.01, \text{ and the number of time steps is } 64$. The best performance is highlighted in bold.}
\label{tab:1deasycase2}
\centering
\scriptsize
\setlength{\tabcolsep}{2pt}
\resizebox{\textwidth}{!}{
\begin{tabular}{|c|cccc|cccc|cccc|cccc|}
\hline
\multirow{2}{*}{$n_1$} & \multicolumn{4}{c|}{\texttt{Iter}} & \multicolumn{4}{c|}{\texttt{Time}(s)} & \multicolumn{4}{c|}{$\texttt{Obj}_{\text{proj}}(10^{-2})$} & \multicolumn{4}{c|}{\texttt{Pres}} \\
\cline{2-17}
 & A & F & G & L & A & F & G & L & A & F  & G & L & A & F & G & L \\
\hline
256 & 2103 & 1624 & 1950 & \textbf{1002} & {15} & 36 & 14 & \textbf{7} & 5.52368 & 5.50488 & 5.50469 & \textbf{5.50467} & 1.2E-5 & \textbf{1.60E-12} & 3.4E-6 & {2.3E-12} \\
512 & 2446 & 1777 & 2187 & \textbf{1001} & 29 & 109 &  33 & \textbf{13} & 5.50496 & 5.50496 & 5.50472 & \textbf{5.50471} & 7.5E-6 & \textbf{4.65E-12}  & 6.2E-6 & \textbf{6.1E-12} \\
1024 & 2673 & 1782 & 2389 & \textbf{856} & 60 & 185  & 71 & \textbf{22} & 5.50497 & 5.50497  & 5.50473 & \textbf{5.50472} & 1.3E-5 & \textbf{1.84E-11} & 3.8E-5 & \textbf{5.5E-11}  \\
2048 & 2921 & 1783 & 2584 & \textbf{1001} & 145 & 449 & 121 & \textbf{66} & 5.50498 & 5.50498 & 5.50473 & \textbf{5.50473}  & 7.3E-5 & \textbf{9.16E-11} & 7.6E-5 & \textbf{1.4E-10} \\
4096 & 4249 & 1783 & 2780 & \textbf{1002} & 471 & 1185 & 295 & \textbf{133} & 5.50498 & 5.50498 & 5.50473 & \textbf{5.50472} & 9.9E-5 & \textbf{4.99E-10} & 5.3E-5 & \textbf{9.8E-10}  \\
\hline
\end{tabular}}
\end{table}

For both values of $\kappa$ reported in Tables~\ref{tab:1deasycase1} and \ref{tab:1deasycase2}, LADMM consistently outperforms the competing methods in terms of computational time and objective value. 

\subsubsection{Normalized
\texorpdfstring{$\bm{\rho}_0$ and $\bm{\rho}_1$}
{rho0 and rho1}
with Added Positive Constant}\label{subsec: normalizedsingleagent}
In this experiment, we generate $\bm{\rho}_0$ and $\bm{\rho}_1$ by
\begin{equation}\label{eq:rhosetting}
\begin{aligned}
    &(\tilde{{\rho}}_0)_i = g_{1/3, 0.1}\left(\frac{i}{n_1}\right) + 0.1, \quad (\tilde{{\rho}}_1)_i = g_{2/3, 0.1}\left(\frac{i}{n_1}\right) + 0.1, \text{ for all } i = 0, 1, \ldots, n_1,\\
    &\bm{\rho}_0 = \frac{\tilde{\bm{\rho}}_0} {\| \tilde{\bm{\rho}}_0\|_1 }, \quad \bm{\rho}_1 = \frac{\tilde{\bm{\rho}}_1} {\| \tilde{\bm{\rho}}_1\|_1}.\\
\end{aligned}
\end{equation}
The 
results by three compared methods are presented in Table~\ref{tab:case2}. We observe that G-prox  satisfies the stopping criteria before reaching the maximum number of iterations for smaller grid sizes, i.e., 
$n_1\le 512$ {while it does not 
for $n_1 > 512$}.  
When $n_1 \leq 2048$, we see that the performance between ALM and LADMM is comparable as ALM is better in terms of total iterations but LADMM is better in terms of the objective value and primal residual. For $n_1=4096$, we can see clearly that LADMM dominates in all reported metrics.

\begin{table}[htbp]
\caption{Comparison of ALM~\cite{benamou2000computational}, G-prox~\cite{jacobs2019solving}, and LADMM in 1D space with $\bm{\rho}_0, \bm{\rho}_1$ setting in \eqref{eq:rhosetting} $\text{and the number of time steps} = 64$. The best performance is highlighted in bold.}
\label{tab:case2}
\centering
\scriptsize
\setlength{\tabcolsep}{1pt}
\resizebox{0.8\textwidth}{!}{
\begin{tabular}{|c|ccc|ccc|ccc|ccc|}
\hline
\multirow{2}{*}{$n_1$} & \multicolumn{3}{c|}{\texttt{Iter}} & \multicolumn{3}{c|}{\texttt{Time}(s)}  & \multicolumn{3}{c|} {$\texttt{Obj}_{\text{proj}}$} & \multicolumn{3}{c|}{\texttt{Pres}} \\
\cline{2-13}
 & A  & G & L & A  & G & L & A  & G & L & A  & G & L\\
\hline
256 & \textbf{175}  & 5427 & {304} & \textbf{1.2}  & 27.2 & {3.1}  & 1.86057E-4  & 1.86046E-4 & \textbf{1.86045E-4} & 9.9E-5 & 9.9E-5 & \textbf{1.1E-13}  \\
512 & \textbf{261}  & 9315 & {276} & \textbf{3.5} & 82.9 & {4.2}  & 9.30438E-5  & 9.30255E-5 & \textbf{9.30254E-5} & 9.9E-5  & 9.9E-5 & \textbf{1.5E-13} \\
1024 & \textbf{205}  & 1E4 & {248} & \textbf{5.6}  & 201.1 & {6.4}  & 4.65221E-5 & 4.65137E-5 & \textbf{4.65130E-5} & 9.9E-5 & 3.6E-4 & \textbf{2.4E-13} \\
2048 & \textbf{148}  & 1E4 & {222} & \textbf{8.8}  & 342.3 & {12.6}  & 2.32583E-5  & 2.32626E-5 & \textbf{2.32566E-5} & 9.9E-5  & 6.9E-4 & \textbf{4.7E-13} \\
4096 & 210 & 1E4 & \textbf{198} & 26.5 & 732.7 & \textbf{22.9}   & 1.16585E-5 & 1.16529E-5 & \textbf{1.16283E-5} & 9.9E-5 & 8.2E-3 & \textbf{6.3E-13} \\
\hline
\end{tabular}}
\end{table}

\begin{remark}
    With $\kappa = 0.1$ and $\kappa = 0.01$ in~\eqref{eq:rhosetting_notnormal}, the gradient Lipschitz constant of \eqref{eq: maincostfunction} is on the order of $10^3$ and $10^6$, respectively.
In the $\bm{\rho}_0, \bm{\rho}_1$ setting in \eqref{eq:rhosetting} and \eqref{eq:rhosetting_notnormal_case3}, the minimum density value is on the order of $10^{-8}$, which leads to extremely large gradient Lipschitz constant of~\eqref{eq: maincostfunction}. In this regime, FISTA fails to converge due to a lower bound on the stepsize imposed by the line search procedure.
For this reason, FISTA is omitted from the remaining tables.
\end{remark}

\subsubsection{Unnormalized
\texorpdfstring{$\bm{\rho}_0$ and $\bm{\rho}_1$}
{rho0 and rho1}
with No Added Positive Constant}\label{subsec:numer_set3_1ag}
In this experiment, we set $\bm{\rho}_0$ and $\bm{\rho}_1$ as follows:
\begin{equation}\label{eq:rhosetting_notnormal_case3}
\begin{aligned}
    &{({\rho}}_0)_i = g_{\frac{1}{3},\,0.1}\left(\frac{i}{n_1}\right), \quad
    {({\rho}}_1)_i = g_{\frac{2}{3},\,0.1}\left(\frac{i}{n_1}\right), 
    ~~\text{ for all } i = 0, 1, \ldots, n_1.
\end{aligned}
\end{equation}
{Compared to the previous numerical experiments}, this setting represents the most challenging case, as we do not add any positive constant to keep the densities uniformly bounded away from zero. We see that the smallest initial and terminal density values are around $10^{-10}$.

For ALM, we set $r=1$, and for G-prox, $(\tau_d, \tau_p)=(1,1)$, both selected via empirical tuning.
We first obtain a numerical approximation of the minimum objective value for the discretized problem; we call it $\mathrm{obj}_{\mathrm{num}}^*$.
To calculate the optimal objective value, we set the stopping tolerance to  $10^{-8}$ and maximum iterations to $3\times 10^4$ and report the results in Table~\ref{table: objstar}. From Table~\ref{table: objstar}, we observe that LADMM achieves the {lowest} objective value and the lowest constraint violation for all values of $n_1$. We also note that G-prox and ALM fail to {reach the stopping tolerance} $10^{-8}$ 
within the maximum number of iterations in all tested values of $n_1$.
\begin{table}[htbp]
\caption{Comparison of ALM~\cite{benamou2000computational}, G-prox~\cite{jacobs2019solving}, and LADMM in 1D space with $\bm{\rho}_0, \bm{\rho}_1$ setting in \eqref{eq:rhosetting_notnormal_case3} and $\text{the number of time steps is } 64$ to find the ``optimal" objective value. The best performance is highlighted in bold.}
\label{table: objstar}
\centering
\scriptsize
\setlength{\tabcolsep}{2pt}
\resizebox{0.8\textwidth}{!}{
\begin{tabular}{|c|ccc|ccc|cc|ccc|ccc|}
\hline
\multirow{2}{*}{$n_1$} & \multicolumn{3}{c|}{\texttt{Iter}} & \multicolumn{3}{c|}{$\texttt{Obj}_{\mathrm{proj}}(10^{-2})$} & \multicolumn{2}{c|}{$| \texttt{Obj} - \texttt{Obj}_{\mathrm{proj}}|$} &  \multicolumn{3}{c|}{\texttt{Pres}} \\
\cline{2-12}
 & A  & G & L & A & G & L & A  & G  & A  & G & L\\
 \hline
256 & 3E4 & 3E4 & \textbf{10739}   & 5.54576  & 5.54301 & \textbf{5.54298}  & 4E-14 & 6E-8 & 3.02E-8 & 2.43E-2 & \textbf{2.44E-13}\\
512 & 3E4   & 3E4 & \textbf{15553} & 5.54601   & 5.54340 & \textbf{5.54300}  & 4E-11 &  3E-6 & 8.12E-9 & 1.76 & \textbf{1.13E-13}\\
1024 & 3E4  & 3E4 & \textbf{14050}  & 5.54614 &  5.54395 & \textbf{5.54302}  & 4E-10 & 9E-6 & 2.06E-8 & 4.61 & \textbf{2.27E-13} \\
2048 & 3E4  & 3E4 & \textbf{16176}  & 5.54620 &  5.54379 & \textbf{5.54313} & 1E-11 & 7E-6 & 1.05E-7 & 7.35 & \textbf{4.54E-13} \\
4096 & 3E4  & 3E4 & \textbf{15766} & 5.54623 &  5.54353 & \textbf{5.54301} & 1E-8  & 1E-6 & 9.20E-5 & 12.2 & \textbf{1.05E-12} \\
\hline
\end{tabular}}
\end{table}

 We use the { best objective values obtained by LADMM in the $7$-th column of Table~\ref{table: objstar} as $\mathrm{obj}^*_{\mathrm{num}}$ to calculate $|\mathrm{obj} - \mathrm{obj}^*_{\mathrm{num}}|$ and report it in Table~\ref{table: case3}}, {where we set the stopping tolerace to $10^{-4}$ and maximum number of iterations to $10^4$.} From Table~\ref{table: case3}, we see that G-prox and ALM both struggle with the setting of initial and terminal densities in \eqref{eq:rhosetting_notnormal_case3}. In fact, G-prox fails to satisfy the stopping condition within the maximum iteration number with all values of $n_1$ while ALM fails for $n_1 > 1024$.

\begin{table}[htbp]
\caption{Comparison of ALM~\cite{benamou2000computational}, G-prox~\cite{jacobs2019solving}, and LADMM in 1D space with $\bm{\rho}_0, \bm{\rho}_1$ setting in \eqref{eq:rhosetting_notnormal_case3} and $\text{the number of time steps is } 64$. The best performance is highlighted in bold.}
\label{table: case3}
\centering
\scriptsize
\setlength{\tabcolsep}{1.8pt}
\resizebox{0.9\textwidth}{!}{
\begin{tabular}{|c|ccc|ccc|ccc|ccc|}
\hline
\multirow{2}{*}{$n_1$} & \multicolumn{3}{c|}{\texttt{Iter}} & \multicolumn{3}{c|}{\texttt{Time}(s)} & \multicolumn{3}{c|}{$ |\mathrm{obj}-\mathrm{obj}^*_{\mathrm{num}}|$} & \multicolumn{3}{c|}{\texttt{Pres}} \\
\cline{2-13}
 & A  & G & L & A  & G & L & A  & G & L & A  & G & L\\
\hline
256 & 2563 & 1E4 & \textbf{1029} & 17.91  & 66.88 & \textbf{11.98} & 3.24E-5 & 3.96E-7 & \textbf{5.12E-9}  & 9.85E-5 & 2.08 & \textbf{5.68E-14}  \\
512 & 5020   & 1E4 & \textbf{1060} & 53.93  & 133.09 & \textbf{20.37} & 7.15E-5  & 3.99E-7 & \textbf{8.76E-9}  & 9.97E-5 & 0.27 & \textbf{1.14E-13}\\
1024 & 9477  & 1E4 & \textbf{1051} & 218.58  & 241.52 & \textbf{28.92} & 7.36E-5 & 4.11E-7 & \textbf{5.44E-8}  & 9.99E-5 & 10.17 & \textbf{2.27E-13} \\
2048 & 1E4  & 1E4 & \textbf{1067} & 484.20  & 488.09 & \textbf{74.75} & 8.82E-5 & 8.84E-6 & \textbf{7.17E-8} & 4.20E-4 & 3.68 & \textbf{4.55E-13} \\
4096 & 1E4  & 1E4 & \textbf{1066} & 1001.90 & 955.28 & \textbf{149.22} & 8.89E-5 & 9.12E-6 & \textbf{8.14E-8} & 7.15E-4 & 16.40 & \textbf{9.09E-13} \\
\hline
\end{tabular}}
\end{table}

\begin{figure}[htbp]
  \centering

  \begin{minipage}{0.4\textwidth}
    \centering
    \includegraphics[width=\linewidth]{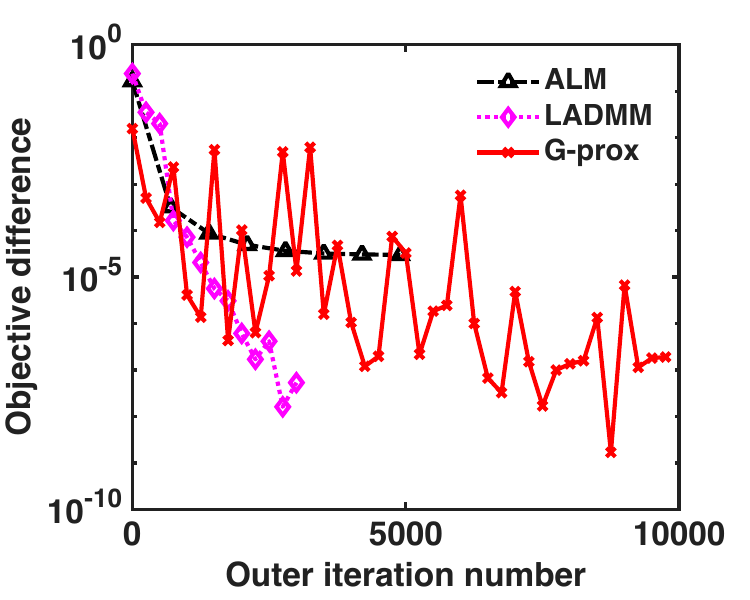}
    {\small $n_1 = 512$}
  \end{minipage} 
  \begin{minipage}{0.4\textwidth}
    \centering
    \includegraphics[width=\linewidth]{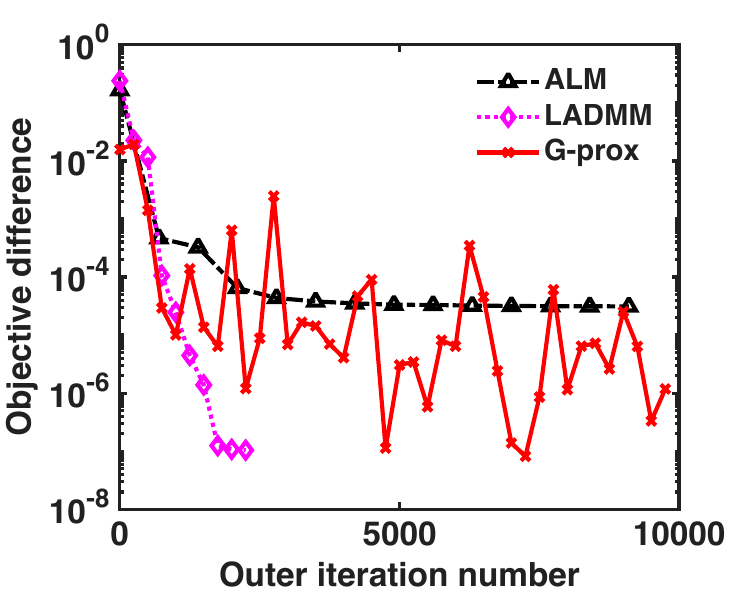}
    {\small $n_1 = 1024$}
  \end{minipage}


  \begin{minipage}{0.4\textwidth}
    \centering
    \includegraphics[width=\linewidth]{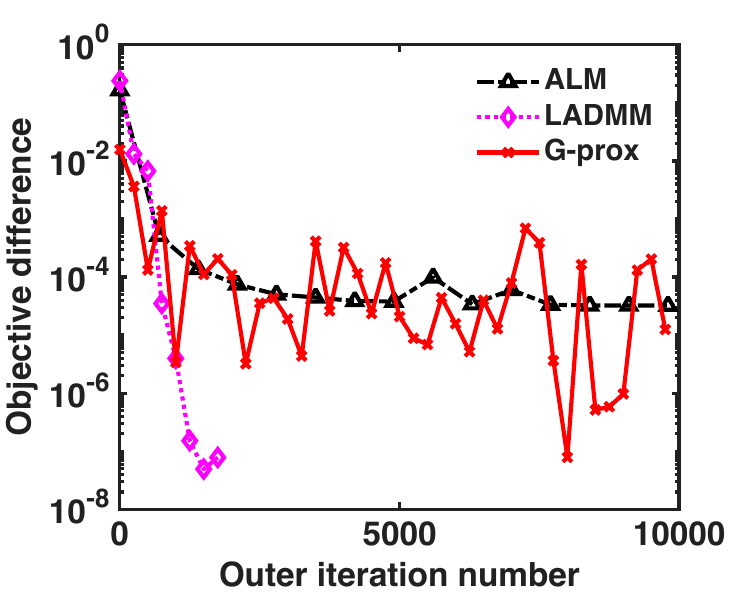}
    {\small $n_1 = 2048$}
  \end{minipage} 
  \begin{minipage}{0.4\textwidth}
    \centering
    \includegraphics[width=\linewidth]{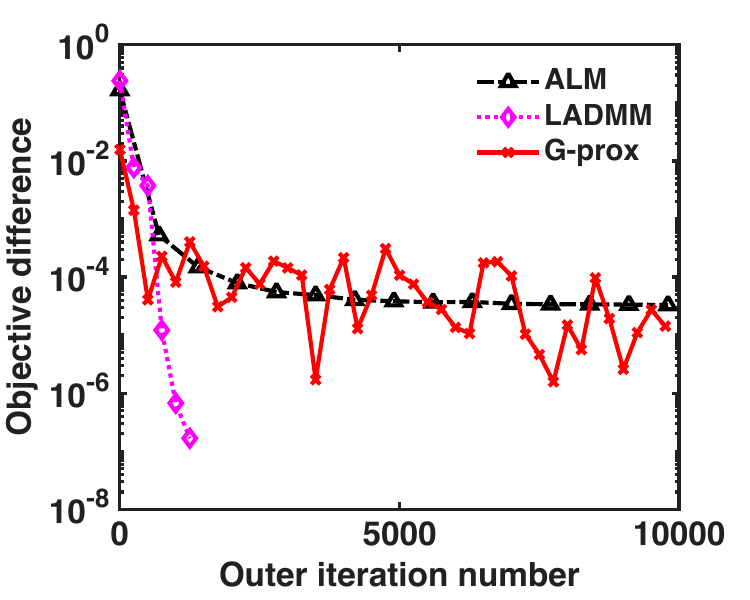}
    {\small $n_1 = 4096$}
  \end{minipage}

  \vspace{0.7em}


  \caption{Objective difference vs. iteration number for ALM~\cite{benamou2000computational}, G-prox~\cite{jacobs2019solving}, and LADMM
           with $\bm{\rho}_0, \bm{\rho}_1$ setting in \eqref{eq:rhosetting_notnormal_case3}, $n_1 \in \{512, 1024, 2048, 4096 \}$, and the number of time steps set to 64.}
  \label{fig:objective_case31D}
\end{figure}

\begin{figure}[htbp]
  \centering

  \begin{minipage}{0.4\textwidth}
    \centering
    \includegraphics[width=\linewidth]{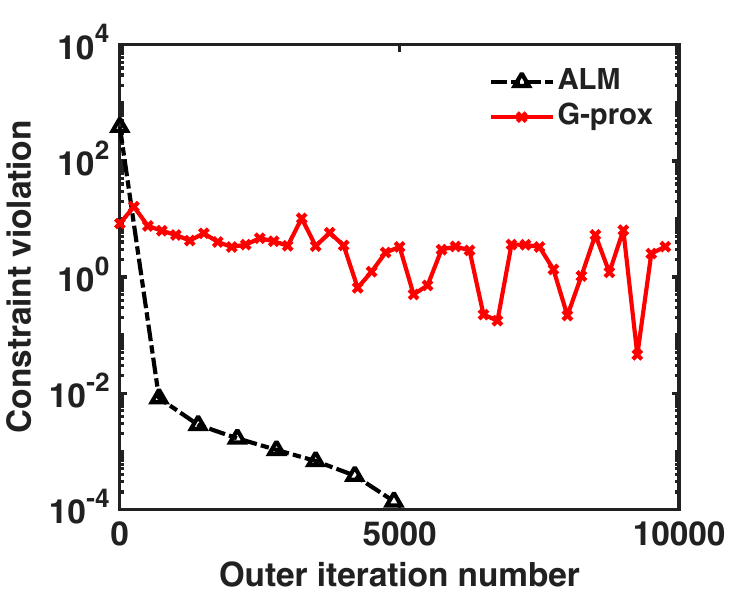}
    {\small $n_1 = 512$}
  \end{minipage}
  \begin{minipage}{0.4\textwidth}
    \centering
    \includegraphics[width=\linewidth]{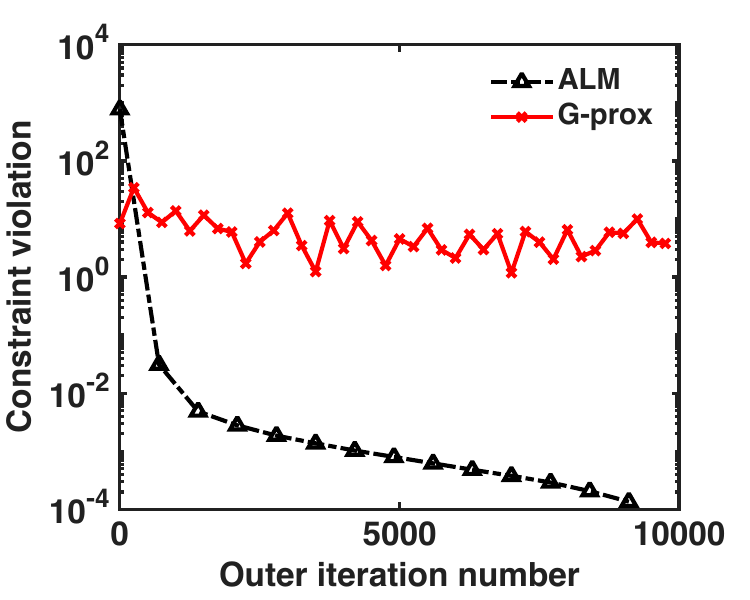}
    {\small $n_1 = 1024$}
  \end{minipage}


  \begin{minipage}{0.4\textwidth}
    \centering
    \includegraphics[width=\linewidth]{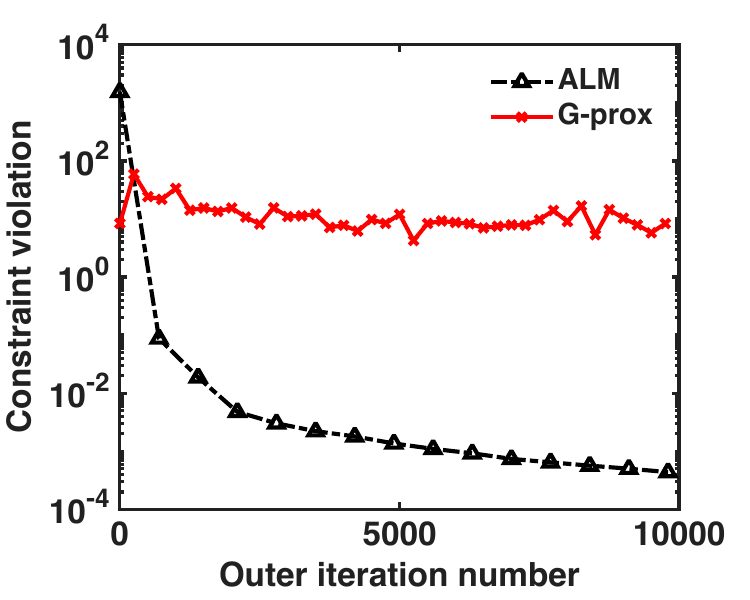}
    {\small $n_1 = 2048$}
  \end{minipage}
  \begin{minipage}{0.4\textwidth}
    \centering
    \includegraphics[width=\linewidth]{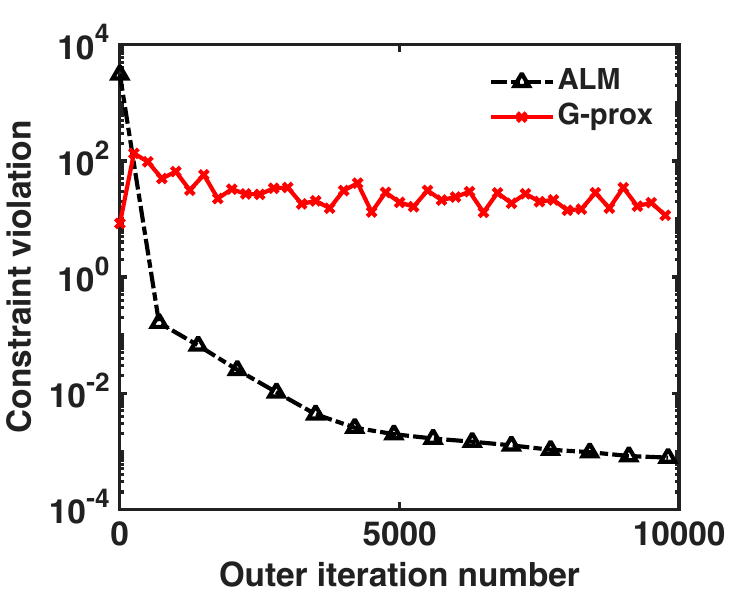}
    {\small $n_1 = 4096$}
  \end{minipage}

  \vspace{0.7em}


  \caption{Constraint violation vs. iteration number for ALM~\cite{benamou2000computational} and G-prox~\cite{jacobs2019solving}
           with $\bm{\rho}_0, \bm{\rho}_1$ setting in \eqref{eq:rhosetting_notnormal_case3}, $n_1 \in \{512, 1024, 2048, 4096 \}$, and the number of time steps set to 64.}
  \label{fig:cons_case31D}
\end{figure}
In Figure \ref{fig:objective_case31D}, {at} every iteration, we report the objective difference defined by
\begin{align}\label{eq:objdiff}
    \mathrm{obj}_{\mathrm{diff}}^k = \big|\, \mathrm{obj}_{\mathrm{num}}^* - \mathrm{obj}^k \,\big|,
\end{align}
where $\mathrm{obj}^k$ is the objective value at the $k$-th iteration and $\mathrm{obj}_{\mathrm{num}}^*$ is the lowest objective value reported in the $7$-th column of Table~\ref{table: objstar}. 
In Figure \ref{fig:cons_case31D}, we report the primal residual by ALM and G-prox. We omit the primal residual for LADMM because it is a projection based method {and its primal residual is always between the order of $10^{-10}$ to $10^{-13}$}. 

\subsubsection{Two dimensional spatial domain}\label{sec:num2d1agent}
In this set of numerical experiments, we consider a two-dimensional spatial domain, i.e., $D=2$. The initial and terminal densities defined in this subsection are therefore matrices.

{We first use 
synthetic data with} 
the initial and terminal densities 
given as follows: for all $i = 0, 1, \ldots, n_1$ and $j = 0, 1, \ldots, n_2$,
\begin{equation}\label{setting: 2drhoconsadd}
\begin{aligned}
    &({\rho}_0)_{i,j} = g_{[1/3, 2/3], [0.1, 0; 0, 0.1]}\left( \frac{i}{n_1}, \frac{j}{n_2}  \right), \\
    &({{\rho}}_1)_{i,j} = g_{[2/3, 1/3], [0.1, 0; 0, 0.1]}\left( \frac{i}{n_1}, \frac{j}{n_2}  \right), 
\end{aligned}
\end{equation}
where $n_1 = 256$ and $ n_2 = 256$.
The density setting in \eqref{setting: 2drhoconsadd} is similar to the one-dimensional setting in \eqref{eq:rhosetting_notnormal_case3}. Since no positive constant is added to ensure that the densities are bounded away from zero, this constitutes a challenging instance of problem~\eqref{prob:ADMM1ag}. The results by three compared methods are presented in Table~\ref{tab:2d1agent}. We observe that LADMM outperforms both ALM and G-prox across all reported metrics. 

\begin{table}[htbb]
\caption{Comparison of ALM~\cite{benamou2000computational}, G-prox~\cite{jacobs2019solving}, and LADMM in 2D space with $\bm{\rho}_0, \bm{\rho}_1$ defined in  \eqref{setting: 2drhoconsadd}, $ n_1 = 256, n_2 = 256$, and the number of time steps set to 64. The best performance is highlighted in bold.}
\label{tab:2d1agent}
\begin{center}
{\small
\begin{tabular}{|c|c|c|c|c|} 
\hline
           & \texttt{Iter} & Time (s) & $\texttt{Obj}_{\mathrm{proj}}$ & \texttt{Pres}\\
          \hline
          ALM & 5564 & 4413 & 0.12136 & 9.23E-04
         \\
         LADMM & \textbf{1787} & \textbf{3298} & \textbf{0.10911} & \textbf{5.19E-13}
         \\
         G-prox & 1E4 & {6686} & 14.6445 & 48.4
         \\
         \hline
\end{tabular}
}
\end{center}

\end{table}

Second, to demonstrate the performance of our method on more complex distributions, we apply it to image data, where the image intensity is interpreted as a probability density.
In this experiment, we want to map initial and terminal densities represented by grayscale images, 
given by the ones at $t_0=0$ and $t_5=1$ in Figure~\ref{fig:centaur_man}. More specifically, let $\mathrm{Im}_0$ and $\mathrm{Im}_1 \in \mathbb{R}^{n_1 \times n_2}$ denote the pixel values of the leftmost and rightmost images in Figure~\ref{fig:centaur_man}. The initial and terminal densities are defined by
\begin{align}\label{setting:images}
(\rho_0)_{i,j} = 10^4 \frac{(\mathrm{Im}_0)_{i,j}}{\|\mathrm{Im}_0\|_1} + 0.1, \quad
(\rho_1)_{i,j} = 10^4 \frac{(\mathrm{Im}_1)_{i,j}}{\|\mathrm{Im}_1\|_1} + 0.1, \forall (i,j),
\end{align}
where $\|\cdot\|_1$ denotes the sum of the absolute values of all entries. {For all compared methods, we set the stopping tolerance to $10^{-6}$, the maximum number of iterations to $10^4$, and $n_0=20, n_1 = n_2 = 256$. The results are presented in Table~\ref{tab:2d1agent_images}.}
The snapshots of evolution at intermediate times {by LADMM} are  shown in Figure~\ref{fig:centaur_man}. {From Table~\ref{tab:2d1agent_images}, we can see again that the performance of LADMM is better than both G-prox and ALM with respect to all the reported metrics.}

\begin{figure}[htbp]
\centering
\resizebox{\linewidth}{!}{
\begin{tabular*}{\linewidth}{@{\extracolsep{\fill}}cccccc}
$t_0=0$ & $t_1=0.2$ & $t_2=0.4$ & $t_3=0.6$ & $t_4=0.8$ & $t_5=1$ \\[4pt]
\includegraphics[width=.13\linewidth]{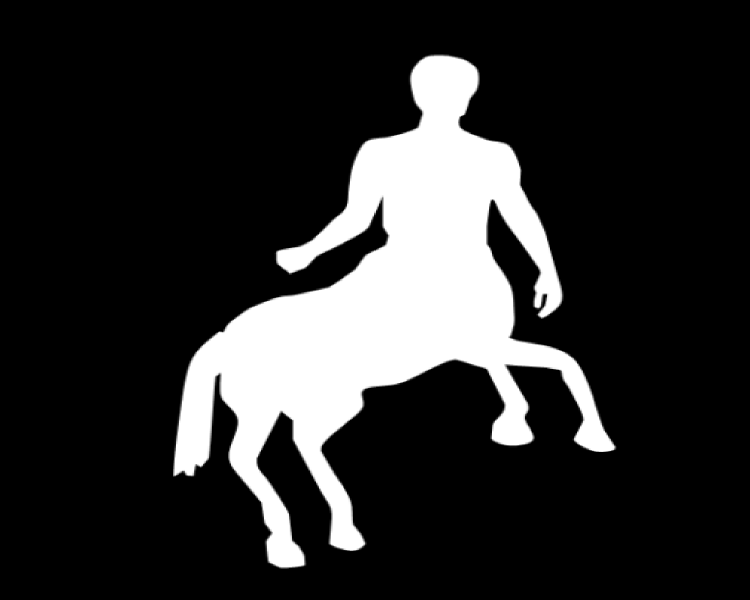} &
\includegraphics[width=.13\linewidth]{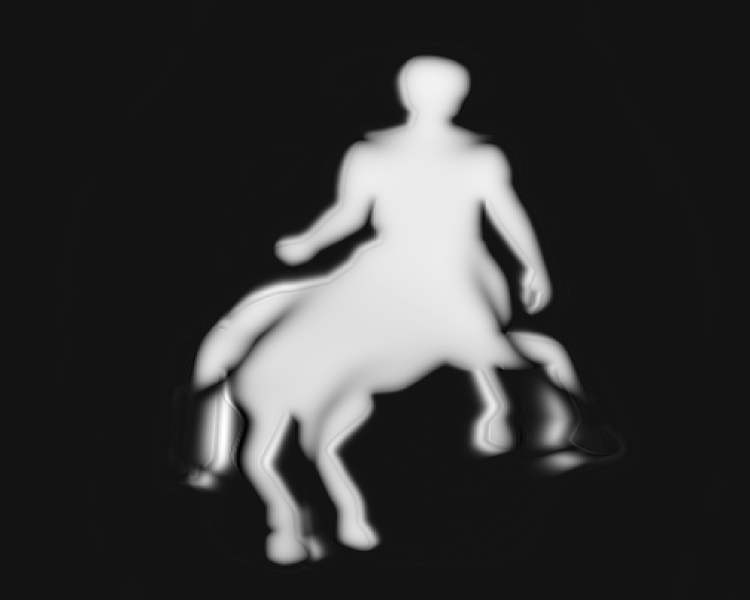} &
\includegraphics[width=.13\linewidth]{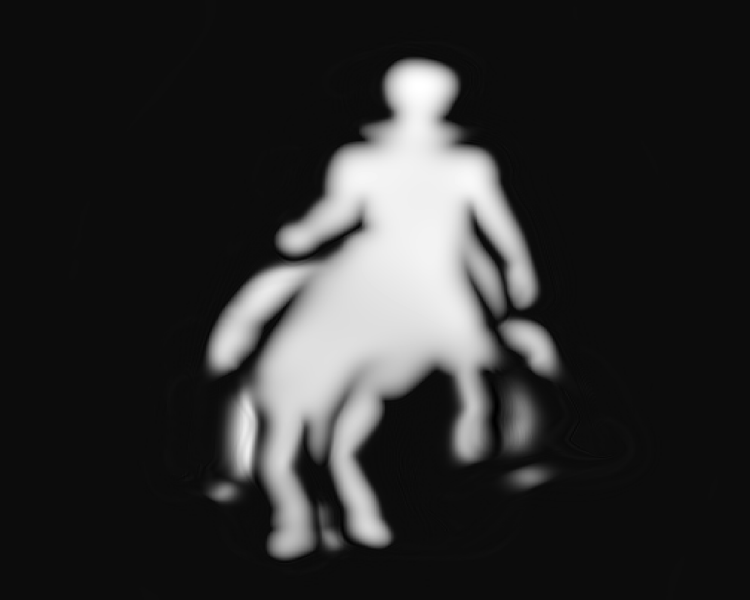} &
\includegraphics[width=.13\linewidth]{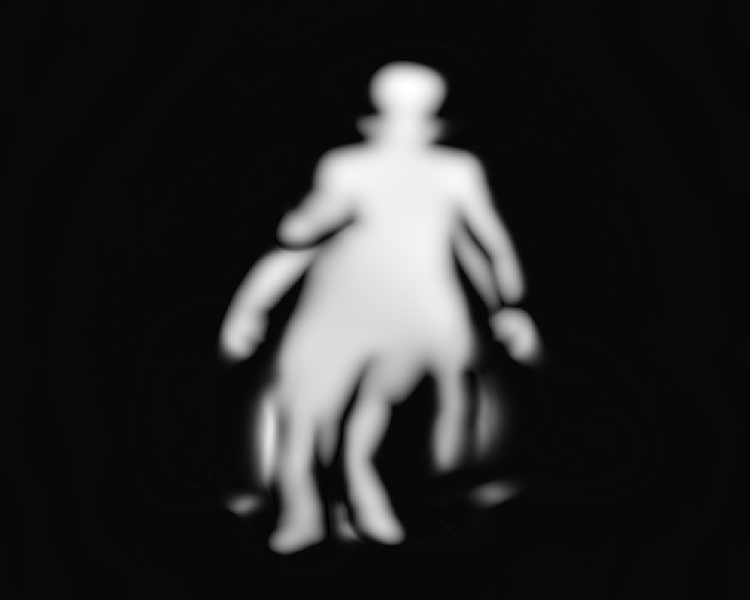} &
\includegraphics[width=.13\linewidth]{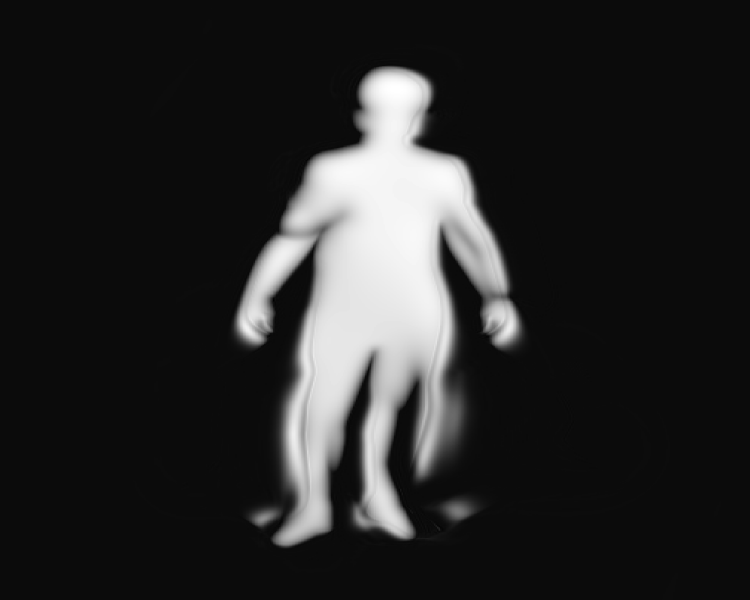} &
\includegraphics[width=.13\linewidth]{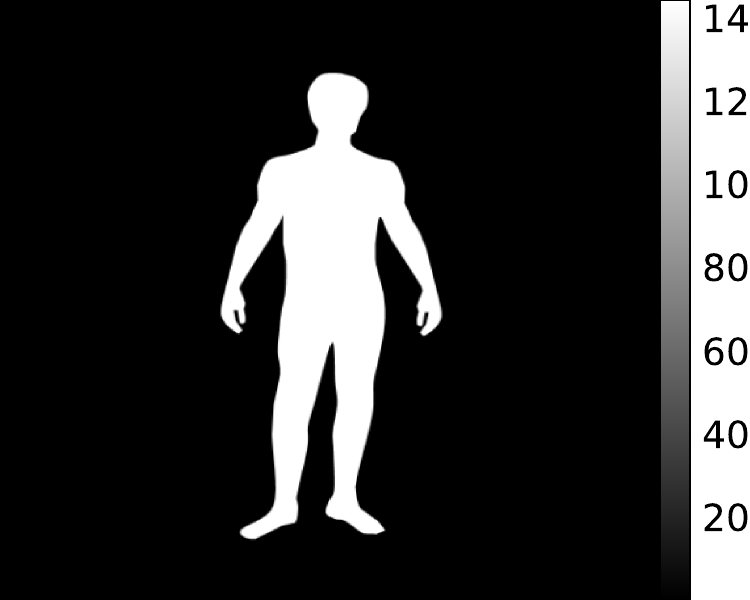}
\end{tabular*}
}
\caption{Evolution of images by LADMM with $n_0=20, n_1=n_2=256$, $ \beta_0 = 10^{-4}$ at times $t=0, 0.2, 0.6, 0.8, 1$ from left to right. 
}
\label{fig:centaur_man}
\end{figure}

\begin{table}[htbb]
\caption{Comparison of ALM~\cite{benamou2000computational}, G-prox~\cite{jacobs2019solving}, and LADMM in 2D space with $\bm{\rho}_0, \bm{\rho}_1$ defined in  \eqref{setting:images}, $ n_1 = 256, n_2 = 256$, and the number of time steps set to 20. The best performance is highlighted in bold.}
\label{tab:2d1agent_images}
\begin{center}
{\small
\begin{tabular}{|c|c|c|c|c|} 
\hline
           & \texttt{Iter} & Time (s) & $\texttt{Obj}_{\mathrm{proj}}$ & \texttt{Pres}\\
          \hline
          ALM & 1E4 & 16850 & 3.53E-2 & 1.11E-03
         \\
         LADMM & \textbf{6536} & \textbf{11786} & \textbf{3.32E-2} & \textbf{6.11E-11}
         \\
         G-prox & 1E4 & {16751} & 1.10143 & 3.52
         \\
         \hline
\end{tabular}
}
\end{center}
\end{table}

\subsection{Numerical results in a distributed setting}\label{subsec: multiagentnumerical}
In this subsection, we solve \eqref{prob:ADMM1ag} in a distributed setting. Similar to Section \ref{sec:numexp1ag}, {we conduct} experiments in {the setting of both} one- and two-dimensional spatial domain.



\subsubsection{1D spatial domain}

In this setting, we set the initial and terminal densities by \eqref{eq:rhosetting_notnormal_case3}. 
%
%
The experiment is performed for different numbers of agents ranging from 1 to 12. 
We present the results with two different gridsize settings: $n_0 = 1024, n_1 = 32768$ and $n_0 = 512, n_1 = 65536$.
The results are presented in Figure~\ref{fig:combined_runs}. 
LADMM converges {to the same level of accuracy} within 1060 iterations for both gridsize settings. 
From Figure~\ref{fig:combined_runs}, {we observe that LADMM achieves close to $7\times$ speedup when scaling from $1$ to $12$ agents while reducing the storage burden by distributing the variables across multiple agents.}

\begin{figure}[htbp]
\centering
    \includegraphics[width=0.7\linewidth]{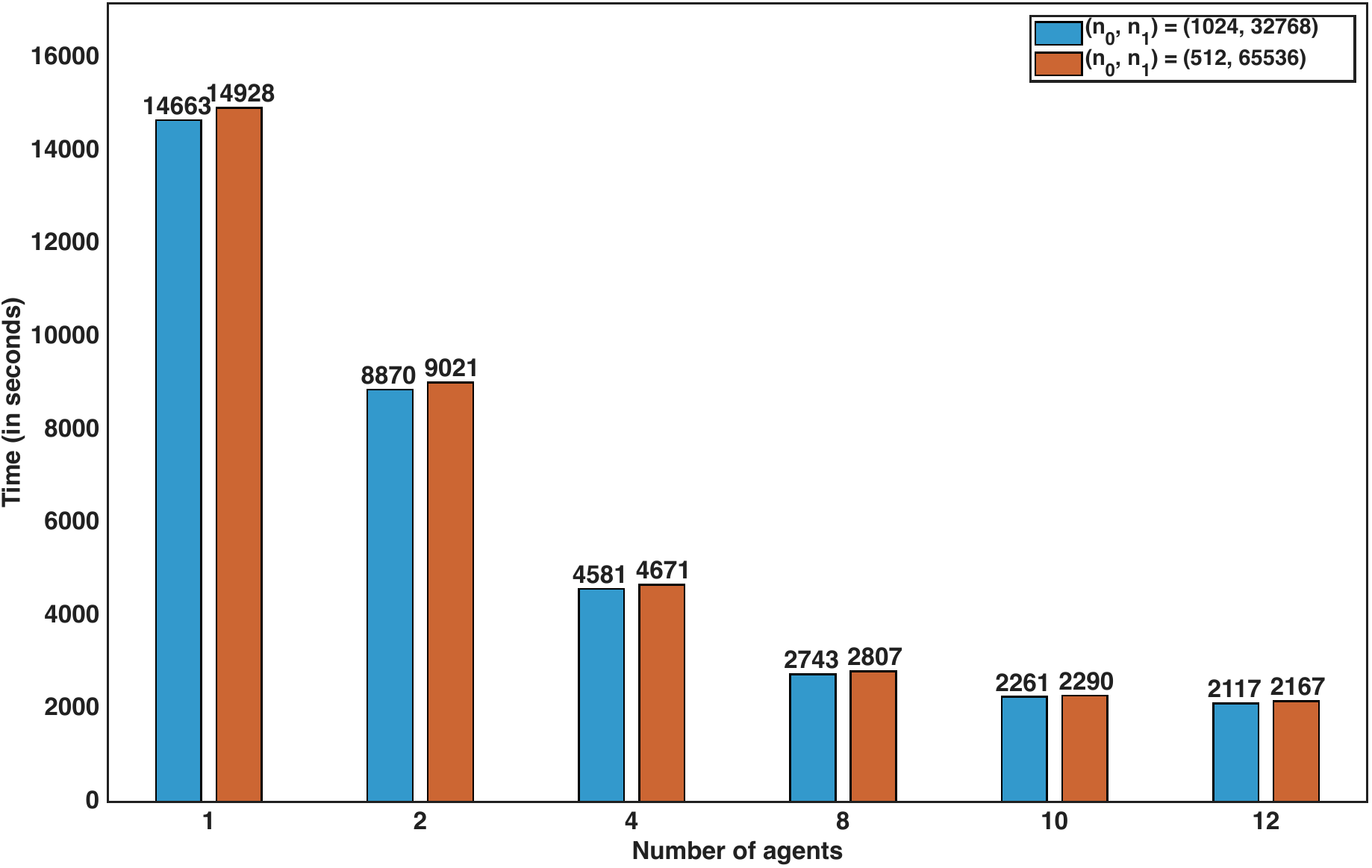}
\caption{Time taken vs. the number of agents by LADMM with dentity setting~\eqref{eq:rhosetting_notnormal_case3}, $N \in \{1, 2, 4, 8, 10, 12 \}$, $\beta_0=10^{-4}$, and $(n_0, n_1) = (1024, 32768)$ and $(n_0, n_1) = (512, 65536)$ represented by blue and red respectively.}
\label{fig:combined_runs}
\end{figure}

\subsubsection{2D spatial domain}\label{num:mapimages}
In this experiment, we examine the evolution of densities when the initial and terminal densities are represented by grayscale images.
We consider the mapping between images of a chalice and a glass, i.e., the leftmost and rightmost images in Figure~\ref{fig:chaliceGlass}.\footnote{The images used for $\bm{\rho}_0$ and $\bm{\rho}_1$ in Figure~\ref{fig:chaliceGlass} are modified from \url{https://www.metmuseum.org/art/collection}.} {The density setting follows the setting in~\eqref{setting:images}.} The snapshots of intermediate evolution at different times are presented in Figure~\ref{fig:chaliceGlass}.

\begin{figure}[htbp]
\centering
\begin{tabular*}{\linewidth}{@{\extracolsep{\fill}}cccccc}
$t_0=0$ & $t_1=0.2$ & $t_2=0.4$ & $t_3=0.6$ & $t_4=0.8$ & $t_5=1$ \\[4pt]

\includegraphics[width=0.13\linewidth]{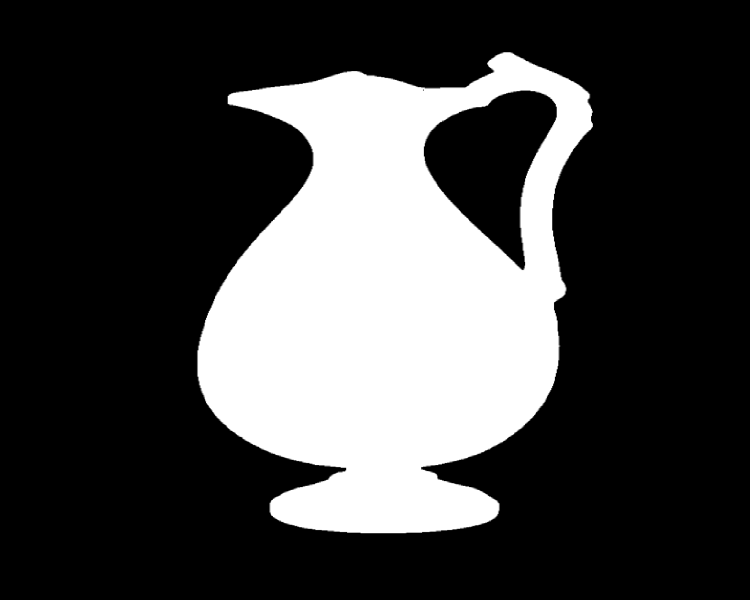} &
\includegraphics[width=0.13\linewidth]{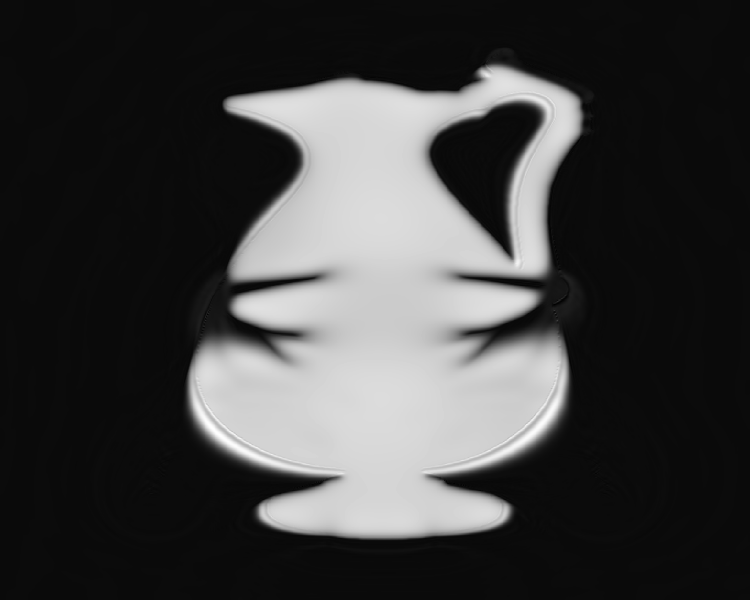} &
\includegraphics[width=0.13\linewidth]{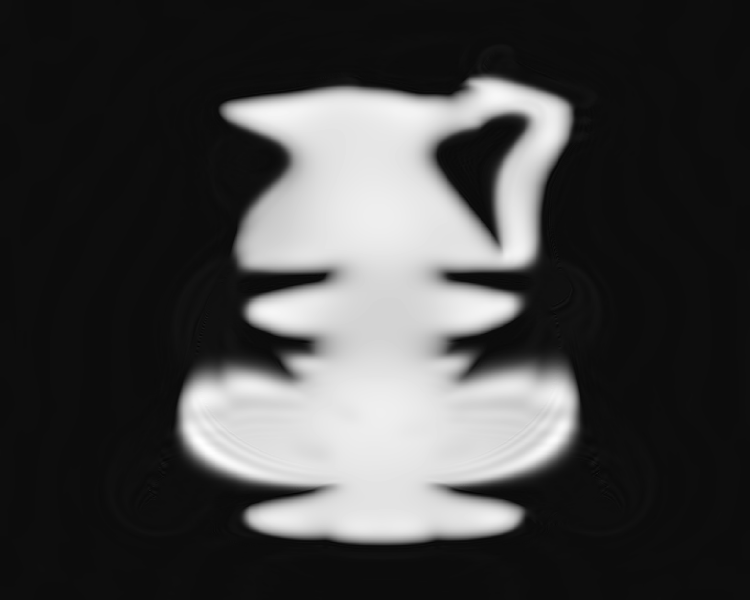} &
\includegraphics[width=0.13\linewidth]{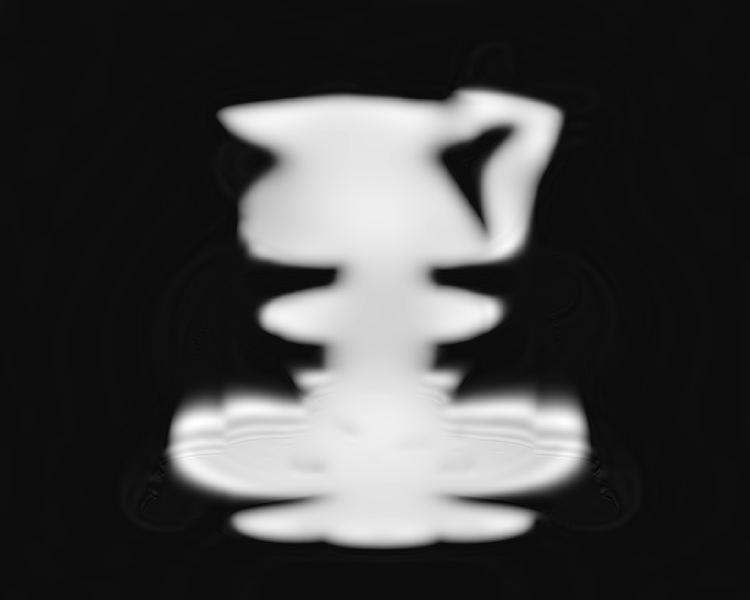} &
\includegraphics[width=0.13\linewidth]{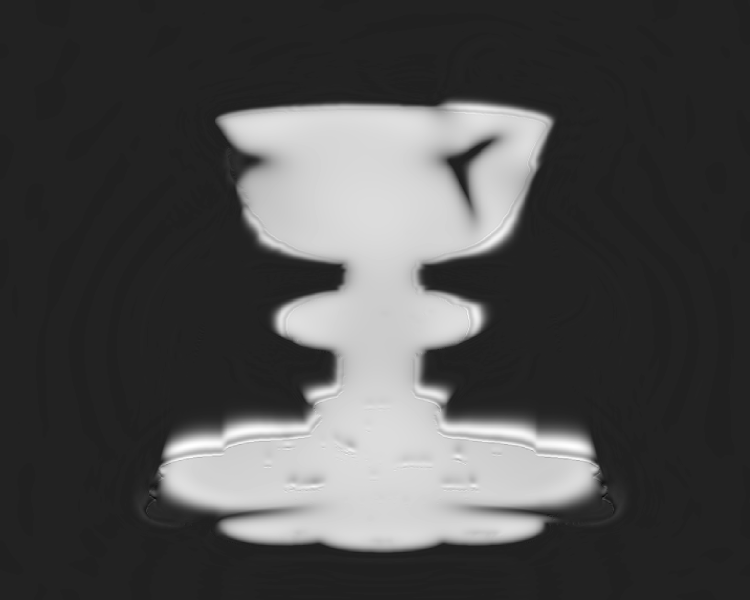} &
\includegraphics[width=0.13\linewidth]{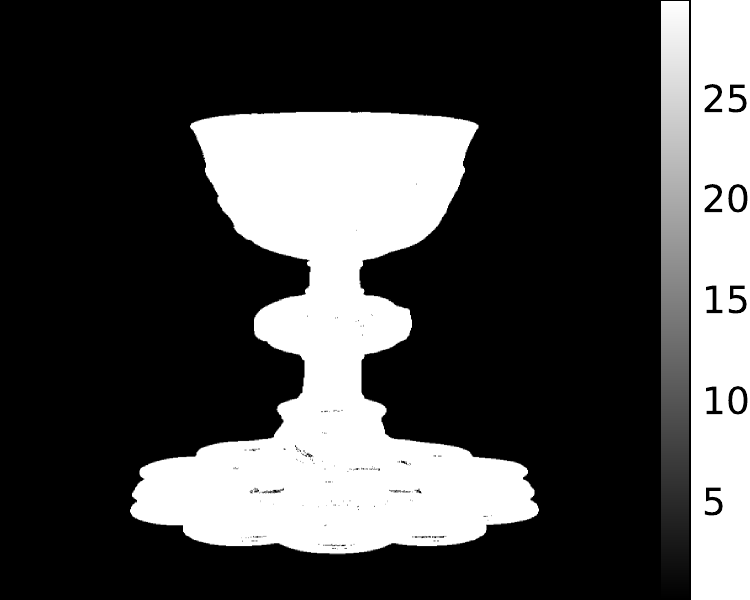}
\end{tabular*}
\caption{Evolution of images by LADMM with $n_0 = 20, n_1 = n_2 = 784, \beta_0 = 10^{-4}$ at times $t=0, 0.2, 0.4, 0.6, 0.8, 1$ from left to right.}
\label{fig:chaliceGlass}
\end{figure}

Furthermore, this experiment is run with varying number of agents from 1 to 12 (the result presented in Figure~\ref{fig:chaliceGlass} is run with $N= 12$). The corresponding running times for different numbers of agents are presented in Figure~\ref{fig:nx864ny864}, from which 
{we observe that LADMM achieves about $6\times$ speedup when scaling from $1$ to $12$ agents.}

\begin{figure}[htbp]
\centering
    \includegraphics[width=0.6\linewidth]{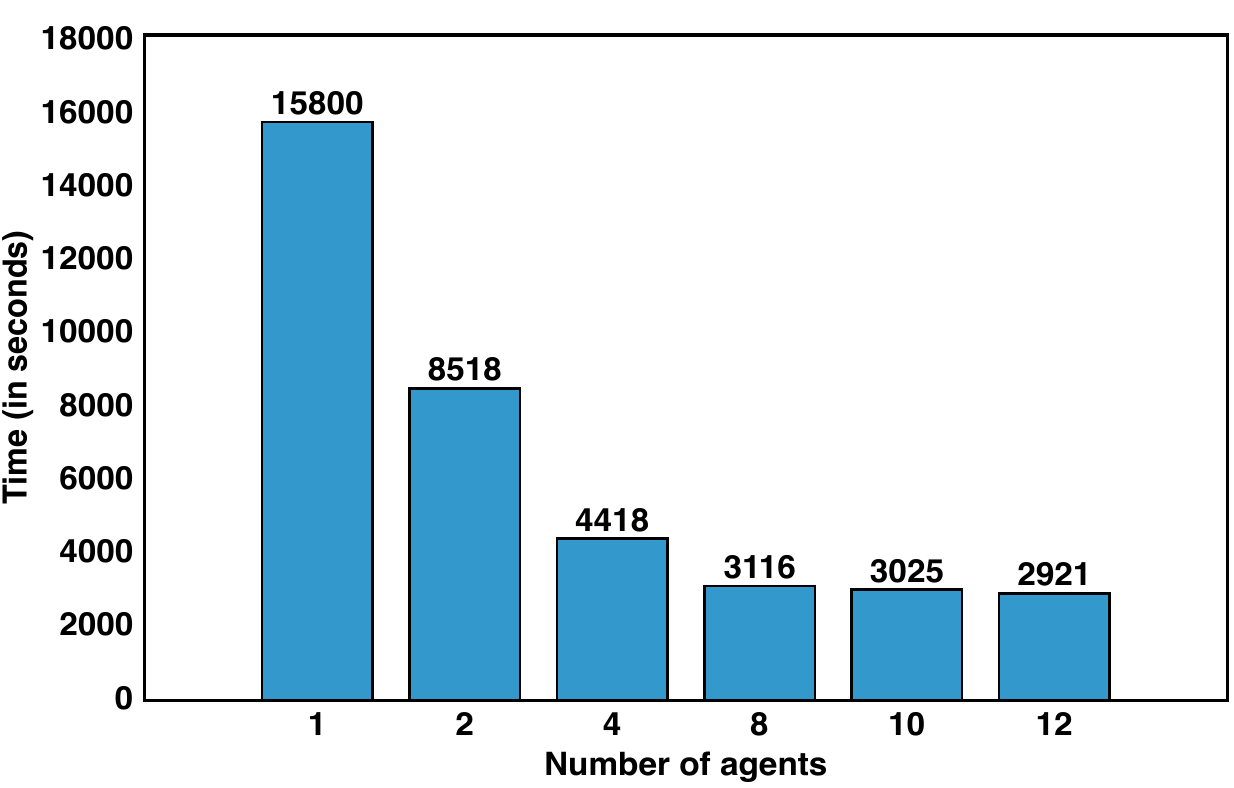}
\caption{Time taken vs. the number of agents for LADMM during image mapping in Figure~\ref{fig:chaliceGlass} with $n_0 = 20, n_1 = 784, n_2 = 784$, $\beta_0=10^{-4}$, and $N \in \{1, 2, 4, 8, 10, 12 \}$. Total iteration taken is $1090$.}
\label{fig:nx864ny864}
\end{figure}

\section{Conclusion}\label{sec:conclusion}
We have proposed an LADMM algorithm 
with an exact proximal mapping for solving the dynamic optimal transport problem~\eqref{problem: main} in both nondistributed and distributed settings. The proposed approach improves upon existing methods in two key aspects. First, it provides a robust and efficient solver for problems in which the initial and/or terminal densities approach zero, a regime where existing methods often suffer from slow convergence or instability. Second, it improves storage and computational efficiency through a distributed formulation that partitions the space--time variables across multiple agents.

We have conducted extensive numerical experiments in one- and two-dimensional spatial settings, considering both nondistributed and distributed implementations and varying levels of problem difficulty. The results demonstrate that the proposed method consistently outperforms existing approaches in terms of accuracy, 
computational time, and iteration count. In several challenging cases, our method successfully solves problem instances for which competing methods fail to converge. In the distributed setting, the proposed method achieves a \(6\times\) speedup when scaling from \(1\) to \(12\) agents, illustrating its potential for large-scale dynamic OT computations.

\bibliographystyle{siamplain} 
\bibliography{optim}

\appendix 

\section{proof of Lemma~\ref{lemma:proxmain}}

\begin{proof}
We first show that the lower-level set is closed for any $\alpha \geq 0$ to guarantee lower semicontinuity (l.s.c) 
of $F$. Define the lower level set of $F$ as
\[
L_{\alpha} = \{(\rho, \vm) \mid F(\rho, \vm) \leq \alpha\}.
\]
 We treat $\alpha = 0$ and $\alpha > 0$ cases separately and have 
\[
L_{\alpha} =
\begin{cases}
\big\{(\rho, \mathbf{0}) \mid \rho \geq 0\big\}, & \text{ if }\alpha = 0, \\[6pt]
\{(0,\vzero)\} \cup \big\{ (\rho, \vm) \mid
\frac{\|\vm\|^2}{2\rho} \leq \alpha,\ \rho > 0, 
\big\}, & \text{ if } \alpha > 0.
\end{cases}
\]

It is easy to see that for both cases of $\alpha$, the lower level set is closed and hence the proximal mapping of $F$ is well defined as 
\begin{align}\label{eq: proxdef}
    \mathrm{Prox}_{F, \gamma}(\widehat{\rho}, \widehat{\vm}) = \argmin_{\rho, \vm} F(\rho, \vm) + \frac{1}{2\gamma} \left[ (\rho - \widehat{\rho})^2 + \|\vm - \widehat{\vm}\|^2 \right].
\end{align}

We let \((\overline{\rho}, \overline{\vm})\) be the solution of \eqref{eq: proxdef}. Then, using the optimality condition, we get
\begin{align}\label{eq: optcondproxmap}
\mathbf{0} \in \partial F(\overline{\rho}, \overline{\vm}) + \frac{1}{\gamma} \begin{bmatrix} \overline{\rho} - \widehat{\rho} \\ \overline{\vm} - \widehat{\vm} \end{bmatrix}.
\end{align}
From the definition of $F$ in \eqref{eq: maincostfunction}, we see that when \( \overline{\rho} = 0 \), then \( \overline{\vm} = 0 \). Suppose \( \boldsymbol{\xi} \in \partial F(0, \mathbf{0}) \), where $\boldsymbol{\xi}$ is a subgradient of $F$ at $(0,\mathbf{0})$. Since $F$ is convex and lower semicontinuous, its subdifferential at $(0,\mathbf{0})$ is characterized as follows. 
\begin{align}
&    F(\rho, \vm) \geq F(0, \mathbf{0}) + \left\langle \boldsymbol{\xi}, \begin{bmatrix}
    \rho \\
    \vm
\end{bmatrix} \right\rangle, \quad \forall (\rho, \vm) \in \text{dom}(F), \notag \\
&\Leftrightarrow \frac{\|\vm\|^2}{2\rho} \geq \rho \xi_1 + \boldsymbol{\xi_2}^\top \vm, \quad \forall \rho > 0, \, \forall \vm \in \mathbb{R}^D, \notag \\
&\Leftrightarrow \|\vm\|^2 \geq 2\rho^2 \xi_1 + 2 \rho \boldsymbol{\xi_2}^\top \vm, \quad \forall \rho > 0, \, \forall \vm \in \mathbb{R}^D \label{eq:proxcond1},
\end{align}
where $\xi_1$ is the first component of $\boldsymbol{\xi}$ and $\boldsymbol{\xi_2}$ is the rest. We present the following two claims.

Claim 1: \( \xi_1 \leq 0 \). This can be proved by contradiction.
 If \( \xi_1 > 0 \), the condition \eqref{eq:proxcond1} is violated at \( \vm = 0 \), \( \rho > 0 \).

Claim 2: \( {\|\boldsymbol{\xi_2}\|^2} \leq -2\xi_1 \). This can be shown by looking at \eqref{eq:proxcond1}. We have,  
\begin{align}\label{eq:mcube}
     \|\vm\|^2 - 2\rho \boldsymbol{\xi_2}^\top \vm - 2\rho^2 \xi_1 \geq 0, \quad \forall \vm \in \mathbb{R}^D,
\end{align}
which indicates
\begin{align*}
    4\rho^2 \|\boldsymbol{\xi_2}\|^2 + 8\rho^2 \xi_1 \leq 0 \iff \|\boldsymbol{\xi_2}\|^2 \leq -2\xi_1.
\end{align*}
Combining Claims 1 and 2, we get
\begin{align}\label{eq: subdiff00}
\partial F(0,\vzero) = \{(\xi_1, \boldsymbol{\xi_2}) \mid \xi_1 \leq 0, \|\boldsymbol{\xi_2}\|^2 \leq -2\xi_1\}.
\end{align}
Hence from \eqref{eq: optcondproxmap} and 
\eqref{eq: subdiff00}, we have that $(0,\mathbf{0})$ is an optimal point if and only if $\|\widehat{\vm}\|^2 \le -2 \widehat{\rho}$.

When \( \rho > 0 \), \( F \) is differentiable, and
\begin{align}\label{eq: gradproxfunct}
    \nabla F(\rho, \vm) = 
\begin{bmatrix}
-\frac{\|\vm\|^2}{2\rho^2} \\
\frac{\vm}{\rho}
\end{bmatrix}.
\end{align}
Together from \eqref{eq: optcondproxmap} and \eqref{eq: gradproxfunct}, we get 
\begin{align*}
    \mathbf{0} = \begin{bmatrix}
-\frac{\|\overline{\vm}\|^2}{2\overline{\rho}^2} \\[2mm]
\frac{\overline{\vm}}{\overline{\rho}}
\end{bmatrix} + \frac{1}{\gamma} \begin{bmatrix} \overline{\rho} - \widehat{\rho} \\[2mm] \overline{\vm} - \widehat{\vm} \end{bmatrix}.
\end{align*}
Solving the above system for $\overline{\rho}$ and $\overline{\vm}$, we get 
\begin{align}\label{eq: cubicrho}
    &\overline{\rho}^3 + (2\gamma - \widehat{\rho})\overline{\rho}^2 + (\gamma^2 - 2 \gamma \widehat{\rho})\overline{\rho} - \frac{1}{2}\gamma \|\widehat{\vm}\|^2 - \gamma^2\widehat{\rho} = 0, \\ \label{eq: solveform}
    &\overline{\vm} = \frac{\widehat{\vm} \overline{\rho}}{\overline{\rho} + \gamma}.
\end{align}
Thus, we can solve the cubic equation in \eqref{eq: cubicrho} to get $\overline{\rho}$ and plug its value in \eqref{eq: solveform} to get $\overline{\vm}$. This completes the proof.
    \end{proof}

\end{document}